\documentclass[final,reqno,onefignum,onetabnum]{siamltex1213}

\usepackage{subdepth}
\usepackage{overpic}
\usepackage{amsmath,amstext,amsbsy,amssymb,mathdots}
\usepackage{booktabs,bm}
\usepackage{mathrsfs}
\usepackage{tikz,cite}
\usepackage{todonotes}
\usepackage{courier}
\usetikzlibrary{positioning}
\usetikzlibrary{shapes,arrows}
\usepackage[fleqn,tbtags]{mathtools}
\usepackage{amsfonts}
\usepackage{relsize}
\usepackage{xcolor,colortbl}
\usepackage{psfrag}
\usepackage{alltt}
\usepackage{arydshln}
\usepackage{enumitem}

\newcommand{\grads}{\nabla_{\mathcal{S}}}
\newcommand{\diffx}[1]{\dfrac{\tpartial #1}{\partial x}}
\newcommand{\diffy}[1]{\dfrac{\tpartial #1}{\partial y}}
\newcommand{\diffz}[1]{\dfrac{\tpartial #1}{\partial z}}
\newcommand{\even}{\textnormal{{\tiny +}}}
\newcommand{\odd}{\textnormal{{\tiny --}}}
\newcommand{\feven}{f^{\even}}
\newcommand{\fodd}{f^{\odd}}
\newcommand{\fteven}{\tilde{f}^{\even}}
\newcommand{\ftodd}{\tilde{f}^{\odd}}
\newcommand{\eteven}{\tilde{e}^{\even}}
\newcommand{\etodd}{\tilde{e}^{\odd}}
\newcommand{\ceven}{{c}^{\even}}
\newcommand{\codd}{{c}^{\odd}}
\newcommand{\reven}{{r}^{\even}}
\newcommand{\rodd}{{r}^{\odd}}

\newcommand{\meven}{m^{\even}}
\newcommand{\modd}{m^{\odd}}
\newcommand{\keven}{k^{\even}}
\newcommand{\kodd}{k^{\odd}}
\newcommand{\Keven}{K^{\even}}
\newcommand{\Kodd}{K^{\odd}}
\newcommand{\deven}{d^{\even}}
\newcommand{\dodd}{d^{\odd}}
\newcommand{\Minv}{M^{\dagger_{\epsilon}}}
\newcommand{\tpartial}{\partial^{\rm t}}

\title{Computing with functions in spherical and polar geometries I. The sphere}
\author{Alex Townsend\thanks{Department of Mathematics, Massachusetts Institute of Technology, 77 Massachusetts Avenue
Cambridge, MA 02139-4307. (ajt@mit.edu). Supported by NSF grant No.~DMS 1522577.} \and Heather Wilber\thanks{Department of Mathematics, Boise State University, Boise, ID 83725-1555. (heatherwilber@boisestate.edu).  Supported by a grant from the NASA Idaho Space Grant Consortium.} \and Grady B. Wright\thanks{Department of Mathematics, Boise State University, Boise, ID 83725-1555. (gradywright@boisestate.edu).  Supported by NSF grant No.~DMS 1160379.}}

\begin{document}
\maketitle

\begin{abstract}
A collection of algorithms is described for numerically computing 
with smooth functions defined on the unit sphere. Functions are approximated to essentially 
machine precision by using a structure-preserving iterative variant of Gaussian 
elimination together with the double Fourier sphere method. We show that this procedure 
allows for stable differentiation, reduces the oversampling of 
functions near the poles, and converges for certain analytic functions. 
Operations such as function evaluation, differentiation, and integration are 
particularly efficient and can be computed by essentially one-dimensional algorithms. 
A highlight is an optimal complexity direct solver for Poisson's equation on the 
sphere using a spectral method. Without parallelization, we solve Poisson's equation with $100$ million 
degrees of freedom in one minute on a standard laptop.  Numerical results are presented throughout.   
In a companion paper (part II) we extend the ideas presented here to computing with functions on the disk.
\end{abstract}

\begin{keywords}
low rank approximation, Gaussian elimination, functions, approximation theory
\end{keywords}

\begin{AMS}
65D05 
\end{AMS}

\section{Introduction} 
Spherical geometries are universal in computational science and engineering, arising 
in weather forecasting and climate modeling~\cite{Merilees_73_01,Orszag_74_01,Fornberg_95_01,spotz1998fast,LaytonSpotz2003,Cheong2000261,coiffier2011fundamentals},
geophysics~\cite{FlyerWright,FlyerWrightYuen}, and astrophysics~\cite{Bruegmann2013216,BartnikNorton,Tichy2006}. 
At various levels these applications all require the approximation of functions defined on the surface of 
the unit sphere. For such computational tasks, a standard approach is to use longitude-latitude coordinates $(\lambda,\theta)\in[-\pi,\pi]\times[0,\pi]$,
where $\lambda$ and $\theta$ denote the azimuthal and polar angles, respectively. 
Thus, computations with functions on the sphere can be conveniently related to analogous tasks involving functions defined on a
rectangular domain. This is a useful observation that, unfortunately, also has 
many severe disadvantages due to artificial pole singularities introduced by the coordinate transform.  

In this paper, we synthesize a classic technique known as the double Fourier sphere (DFS) 
method~\cite{Fornberg_97_01,Merilees_73_01,Orszag_74_01,Boyd_1978,Yee1980} together with new
algorithmic techniques in low rank function approximation~\cite{bebendorf2000approximation,Chebfun2}. This alleviates 
many of the drawbacks inherent with standard coordinate transforms. Our approximants have several 
attractive properties: (1) no artificial pole singularities, (2) a representation that allows for fast algorithms,
(3) a structure so that differentiation is stable, and (4) an underlying interpolation grid that 
rarely oversamples functions near the poles.  

To demonstrate the generality of our approach we describe a collection of algorithms for performing
common computational tasks and develop a software system for numerically computing with functions on the sphere,
which is now part of Chebfun~\cite{Chebfun}. 
A broad variety of algorithms are then exploited to provide a convenient 
computational environment for vector calculus, geodesic calculations, and the solution of 
partial differential equations. In the second part to this paper we show that these techniques naturally extend to computing 
with functions defined on the unit disk~\cite{diskfun}.  

With these tools investigators can now complete 
many computational tasks on the sphere without worrying about the underlying discretizations. 
Whenever possible, we aim to deliver essentially 
machine precision results by data-driven compression and reexpansion of our underlying function
approximations. Accompanying this paper is publicly available MATLAB code distributed with 
Chebfun~\cite{Chebfun}, which has a new class called spherefun. One may wish to read this paper 
with the latest version\footnote{The spherefun code is publicly available in the Chebfun development branch from \texttt{https://github.com/chebfun/chebfun} and will be included in the next Chebfun release.} of Chebfun downloaded and be ready for interactive exploration. 

Our two main techniques are the DFS method 
and a structure-preserving iterative variant of Gaussian elimination (GE) for low rank 
function approximation\cite{townsend2013gaussian}. The DFS method transforms a function on the sphere
to a function on a rectangular domain that is periodic in both variables, 
with some additional special structure (see Section~\ref{sec:doubleFouriersphere}).
Our GE procedure constructs a structure-preserving 
and data-driven approximation in a low rank representation. The
low rank representation means that many operations, including 
function evaluation, differentiation, and integration, are particularly efficient (see Section~\ref{sec:spherefun}). In 
addition, our representations allow for fast algorithms based on the fast Fourier 
transform (FFT) including a fast Poisson solver based on the Fourier spectral method 
(see Section~\ref{sec:PoissonSphere}).

% Existing methods paragraph, here: 
%There are numerous ideas for computing with functions on the sphere that have many strengths 
%and inherent problems. Here is a selection. 
There are several existing approaches for computing with functions on the sphere.  The following is a selection:
\begin{itemize}[leftmargin=*]
\item {\bf Spherical harmonic expansions:}  Spherical harmonics are the spherical analogue of trigonometric expansions for periodic functions and provide essentially uniform resolution of a function over the sphere~\cite[Chap.~2]{atkinson2012spherical}.  They have major applications in weather 
forecasting~\cite[Ch. 18]{boyd2001chebyshev}, least-squares filtering~\cite{jekeli1996spherical}, and the numerical 
solution of separable elliptic equations.  
\item {\bf Longitude-latitude grids:} See Section~\ref{sec:longitude}.
\item {\bf Quasi-isotropic grid-based methods} Quasi-isotropic grid-based methods, such as those that use the ``cubed-sphere''
(``quad-sphere'')~\cite{Sadourny_72_01,Taylor_97_01}, geodesic (icosahedral) grids~\cite{Baumgardner_85_02}, or equal area ``pixelations''~\cite{Gorski_05_01}, partition the sphere into (spherical) quads, triangles, or other polyhedra, where approximation techniques such as splines or spectral elements can be used.
\item {\bf Mesh-free methods:} Mesh-free methods for the sphere such as radial basis functions (RBFs)~\cite{FASC98} allow for function reconstruction from ``scattered'' nodes on the sphere, which can be arranged in quasi-optimal configurations. These methods have been used in numerical weather prediction and solid earth geophysics~\cite{FlyerWright,FlyerWrightYuen}. 
\item {\bf The double Fourier sphere method:} See Section~\ref{sec:doubleFouriersphere}.
\end{itemize}
%In this paper we require an approximation scheme for functions on the sphere that is 
%highly adaptive and has high accuracy. Therefore, a fast transform is essential so 
%that the computations remain computationally feasible.  
To achieve the goals of this paper and accompanying software, we require an approximation scheme for functions on the sphere that is 
highly adaptive and can achieve $16$-digits of precision.  It is also desirable to have an associated fast transform so that the computations remain 
computationally efficient.  The DFS method combined with low rank approximation, which we develop in this paper, best fits our goals. 

The paper is structured as follows.  We first briefly introduce the software that accompanies this paper.  Following this in Section \ref{sec:existing}, we review the DFS method.  Next, in Section~\ref{sec:lowRankApproximation} we discuss low rank function approximation 
and derive and analyze a structure-preserving GE 
procedure.  We then show how one can perform a collection computational tasks with functions on the sphere using the combined DFS and low rank techniques in Section~\ref{sec:spherefun}.  Finally in Section~\ref{sec:PoissonSphere}, we describe a fast and spectrally accurate method for Poisson equation on the sphere.

\subsection{Software}
There are existing libraries that provide various tools for analyzing functions 
on the sphere~\cite{Simons,LeistedtEtAl_2013,Wieczorek,AdamsSwarztrauber_97}, 
but none that easily allow for exploring functions in an integrated environment.  We have implemented such a package in 
MATLAB and we have made it publicly available as part of Chebfun~\cite{Chebfun}.  The interface to the software is through the creation of spherefun objects.  For example,
\begin{align}
f(\lambda,\theta) = \cos(1 + 2\pi (\cos\lambda\sin\theta + \sin\lambda\sin\theta) + 5 \sin(\pi \cos\theta))
\label{eq:sphereTestFunc}
\end{align}
can be constructed by the MATLAB code:
\begin{alltt}
f = spherefun( @(la,th) cos(1 + 2*pi*(cos(la).*sin(th) +...
                      sin(la).*sin(th)) + 5*sin(pi*cos(th))) )
\end{alltt}
The software also allows for functions to be defined by Cartesian coordinates.  
For example, the following code is equivalent to the above:
\begin{alltt}
f = spherefun( @(x,y,z) cos(1 + 2*pi*(x + y) + 5*sin(pi*z)) )
\end{alltt}
and the output from either of these statements is
\begin{alltt}
f =
     spherefun object: 
         domain        rank    vertical scale
       unit sphere      23            1
\end{alltt}
This indicates that the numerical rank of \eqref{eq:sphereTestFunc} is 23, which is determined using an iterative variant of GE (see Section~\ref{sec:lowRankApproximation}), 
while the vertical scale approximates the absolute maximum of this function.  Once a function has been constructed in spherefun, it can be manipulated 
and analyzed using a hundred or so operations, several of which are discussed 
in Section \ref{sec:spherefun}.  For example, one can perform surface integration (\texttt{sum2}), differentiation (\texttt{diff}), and vector calculus operations (\texttt{div}, \texttt{grad}, \texttt{curl}).

%\section{Existing techniques for computations on the sphere}\label{sec:existing} 
%There are numerous ideas for computing with functions on the sphere that have many strengths 
%and inherent problems. We briefly review a selection of existing techniques.
\section{Longitude-latitude coordinate transforms and the double Fourier sphere method}\label{sec:existing} 
Here, we review longitude-latitude coordinate transforms and the DFS method, which form part of the foundation for our new method. 

\subsection{Longitude-latitude coordinate transforms}\label{sec:coordinateTransforms}\label{sec:longitude}
Longitude-latitude coordinate transforms relate computations with functions on the sphere to 
tasks involving functions on rectangular domains.  The co-latitude coordinate transform is given by
\begin{equation}
x = \cos\lambda\sin\theta, \quad y = \sin\lambda\sin\theta, \quad z = \cos\theta,\qquad (\lambda,\theta)\in[-\pi,\pi]\times[0,\pi],
\label{eq:sphericalCoordinates}
\end{equation}
where $\lambda$ is the azimuth angle and $\theta$ is the polar (or zenith) angle. 
With this change of variables, instead of 
performing computations on $f(x,y,z)$ that are restricted to
the sphere, we can compute with the function $f(\lambda,\theta)$.

However, note that any point of the form $(\lambda,0)$ with $\lambda\in[-\pi,\pi]$ maps to
$(0,0,1)$ by~\eqref{eq:sphericalCoordinates} and hence, the coordinate transform 
introduces an artificial singularity at the north pole.  An equivalent singularity occurs at the south pole for the point $(\lambda,\pi)$.
When designing an interpolation grid for approximating functions, any 
reasonable grid on $(\lambda,\theta)\in[-\pi,\pi]\times[0,\pi]$ 
is mapped to a grid on the sphere that is unnecessarily and severely clustered at the poles. 
Therefore, naive grids typically oversample functions near the poles, resulting 
in redundant evaluations during the approximation process. 

Another issue is that these coordinate transforms do 
not preserve the periodicity of functions defined on the sphere 
in the latitude direction. This means the FFT is not
immediately applicable in the $\theta$-variable.  

\subsection{The double Fourier sphere method}\label{sec:doubleFouriersphere}
The DFS method proposed by Merilees~\cite{Merilees_73_01}, and developed 
further by Orszag~\cite{Orszag_74_01}, Boyd~\cite{Boyd_1978}, Yee~\cite{Yee1980}, and Fornberg~\cite{Fornberg_95_01}, is a simple 
technique that transforms a function on the sphere to one on a rectangular domain, while 
preserving the periodicity of that function in both 
the longitude and latitude directions. First, a function $f(x,y,z)$ on the sphere is 
written as $f(\lambda,\theta)$ using~\eqref{eq:sphericalCoordinates}, i.e.,
\[
 f(\lambda,\theta) = f(\cos\lambda\sin\theta,\sin\lambda\sin\theta,\cos\theta), \qquad (\lambda,\theta)\in[-\pi,\pi]\times[0,\pi]. 
\]
This function $f(\lambda,\theta)$ is $2\pi$-periodic in $\lambda$, but not periodic 
in $\theta$. The periodicity in the latitude direction has been lost. To recover it,  the 
function is ``doubled up'' and a related function on $[-\pi,\pi]\times[-\pi,\pi]$ is defined as 
\begin{equation} 
 \tilde{f}(\lambda,\theta) = 
 \begin{cases} 
  g(\lambda+\pi,\theta), & (\lambda,\theta)\in[-\pi,0]\times[0,\pi],\cr 
  h(\lambda,\theta), & (\lambda,\theta)\in[0,\pi]\times[0,\pi],\cr 
  g(\lambda,-\theta), & (\lambda,\theta)\in[0,\pi]\times[-\pi,0],\cr
  h(\lambda+\pi,-\theta), & (\lambda,\theta)\in[-\pi,0]\times[-\pi,0],
 \end{cases}
\label{eq:BMCsphere}
\end{equation} 
where $g(\lambda,\theta)=f(\lambda-\pi,\theta)$ and $h(\lambda,\theta)=f(\lambda,\theta)$ for $(\lambda,\theta)\in[0,\pi]\times[0,\pi]$.  This doubling up of $f$ to $\tilde{f}$ is referred to as a glide reflection~\cite[\S 8.1]{martin2012transformation}.  Figure~\ref{fig:Altas} demonstrates the DFS method 
applied to Earth's landmasses.  The new function $\tilde{f}$ is $2\pi$-periodic in $\lambda$ and $\theta$, and is constant along the lines $\theta=0$ and $\theta=\pm \pi$, corresponding to the poles.   To compute a particular operation on a function $f(x,y,z)$ on the sphere we use the DFS 
method to related it to a task involving $\tilde{f}$. Once a particular numerical quantity has been calculated 
for $\tilde{f}$, we translate it back to have a meaning for the 
original function $f(x,y,z)$. %The definite integral of $f(x,y,z)$ over the sphere exemplifies this working paradigm (see Section~\ref{sec:spherefun}).

In describing and implementing our iterative Gaussian elimination algorithm for producing low rank approximations of $\tilde{f}$, it is convenient to visualize $\tilde{f}$ in block form, similar to a matrix.  To this end, and with a slight abuse of notation, we depict $\tilde{f}$ as
\begin{equation}
 \tilde{f} = \begin{bmatrix} g & h \\[3pt] {\tt flip}(h) & {\tt flip}(g) \end{bmatrix},   
\label{eq:depictBMC}
\end{equation} 
where ${\tt flip}$ refers to the MATLAB command that reverses the 
order of the rows of a matrix.
%\begin{equation}
% \tilde{f} = \left[\begin{array}{c:c} 
%\hdashline
%g & h \\[3pt]
%\hdashline
% {\tt flip}(h) & {\tt flip}(g) \\[3pt]
%\hdashline
%\end{array}\right],   
%\label{eq:depictBMC}
%\end{equation} 
%where ${\tt flip}$ refers to the MATLAB command that reverses the 
%ordering of the rows of a matrix, and the horizontal dashed lines mark the poles where $\tilde{f}$ is constant. 
The format in~\eqref{eq:depictBMC} shows the structure that we wish to preserve in our algorithm.  We see from~\eqref{eq:depictBMC} that $\tilde{f}$ has a structure close to a $2\times 2$ centrosymmetric matrix, 
except that the last block row is flipped (mirrored). For this reason we say that $\tilde{f}$ in~\eqref{eq:BMCsphere} has block-mirror-centrosymmetric (BMC) structure.
\begin{definition}(Block-mirror-centrosymmetric functions)
Let $a,b\in\mathbb{R}$. A function $\tilde{f}:[-a,a]\times[-b,b]\rightarrow \mathbb{C}$ is a block-mirror-centrosymmetric (BMC) 
function, if there are functions $g,h:[0,a]\times[0,b]\rightarrow \mathbb{C}$ 
such that $\tilde{f}$ satisfies~\eqref{eq:depictBMC}.
% \begin{equation} 
%  \tilde{f}(\lambda,\theta) = 
%  \begin{cases} 
%   g(\lambda,\theta), & (\lambda,\theta)\in[-a,0]\times[0,b],\cr 
%   h(\lambda,\theta), & (\lambda,\theta)\in[0,a]\times[0,b],\cr 
%   g(\lambda-a,-\theta), & (\lambda,\theta)\in[0,a]\times[-b,0],\cr
%   h(\lambda+a,-\theta), & (\lambda,\theta)\in[-a,0]\times[-b,0].
%  \end{cases}
% \label{eq:BMCstructure}
% \end{equation} 
\label{def:BMCfunction}
\end{definition}

Via~\eqref{eq:BMCsphere} every smooth function on the sphere is associated with a smooth BMC function defined on $[-\pi,\pi]\times[-\pi,\pi]$ that 
is $2\pi$-periodic in both variables  (also called {\em bi-periodic}). The converse is not true, since it may be possible to have a smooth BMC function that is bi-periodic but is not constant at the poles, i.e., along the lines $\theta=0$ and 
$\theta=\pi$\footnote{The bi-periodic BMC function $f(\lambda,\theta) = \cos(2\theta)\cos(2\lambda)$ is not constant along $\theta=0$ or $\theta=\pi$.}.  We define BMC functions with this property as follows:
\begin{definition}(BMC-I functions)
A function $\tilde{f}:[-a,a]\times[-b,b]\rightarrow \mathbb{C}$ is a Type-I BMC (BMC-I) function if it is a BMC function and it is constant when its second variable is equal to $0$ and $\pm b$, i.e., $f(\cdot,0) = \alpha$, $f(\cdot,b) = \beta$, and $f(\cdot,-b) = \gamma$.
\label{def:BMCfunctionI}
\end{definition}

For the sphere, we are interested in BMC-I functions defined on $[-\pi,\pi]^2$ that are bi-periodic for which $f$ takes the same constant value when $\theta=\pm \pi$. For 
the disk we are interested in so-called BMC-II functions~\cite{diskfun}.

Our approximation scheme and subsequent numerical algorithms for the sphere preserve BMC-I structure and bi-periodicity of a function 
strictly, without exception. By doing this we can compute with functions on $[-\pi,\pi]^2$ 
while keeping an interpretation on the sphere.  

% Longitude-latitude rectangular grids are a natural and useful numeric representation of functions over the sphere: 
% such grids conform readily to established approximation techniques and prove intuitive in many computing environments.  
% Additionally, the periodicity of these grids in the longitudinal direction can be exploited via the FFT.  
% However, this method .  The \textit{Double Fourier Sphere}, first proposed in the 
% 1970s by Merilees~\cite{merilees1973pseudospectral}, Orszag~\cite{orszag1974fourier} and Boer 
% and Steinberg~\cite{boer1975fourier}, is a classic solution to this problem.  As depicted below, the DFS is formed by 
% reflecting the usual longitude-latitude grid across the South pole and translating the doubled image $\pi$ units in a 
% modular fashion, so that the arcs forming great meridians on the sphere are connected.  The result is a square grid 
% that is $2\pi$-periodic in both directions and possesses a variety of symmetry properties. In addition to allowing 
% bidirectional fast transforms, enforcing these symmetry and periodicity constraints ensures smoothness across the 
% poles, even in low-resolution representations. \textcolor{red}{(I'm thinking here about Dr. Wright's images from 
% the slides that show how chebfun's low resolution approximations are off at the poles. Is this a digression or should 
% we include? )}

The DFS method has been used since the 1970's in numerical weather prediction~\cite{Merilees_73_01,Orszag_74_01,Fornberg_95_01,spotz1998fast,LaytonSpotz2003,Cheong2000261}, it has recently found its way to the 
computation of gravitational fields near black holes~\cite{Bruegmann2013216,BartnikNorton,Tichy2006} and to 
novel space-time spectral analysis~\cite{sun2014contrasting}.

\begin{figure} 
 \centering
 \begin{minipage}{.49\textwidth} 
 \centering
%   \begin{overpic}[width=.68\textwidth,trim=50 50 50 50,clip]{Earth1}
%   \end{overpic}
   \begin{overpic}[width=.5\textwidth]{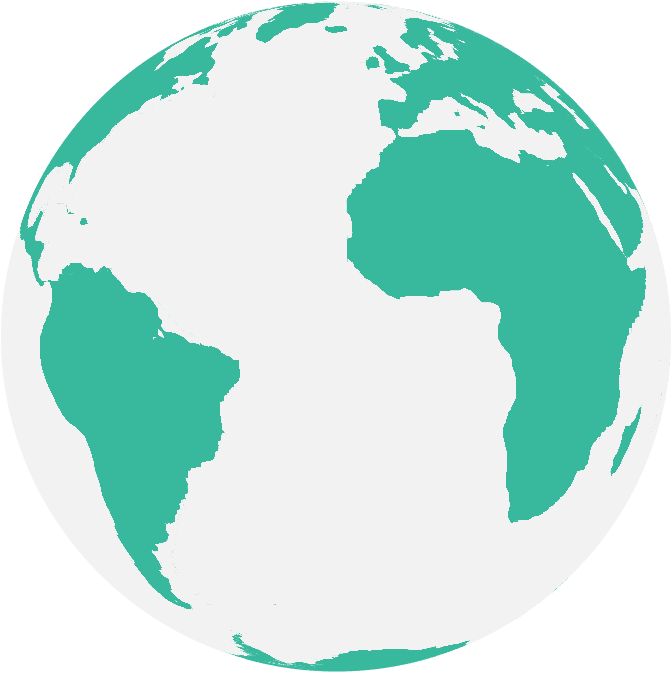}
   \end{overpic}
   
    \begin{overpic}[width=.95\textwidth,trim=30 50 30 60,clip]{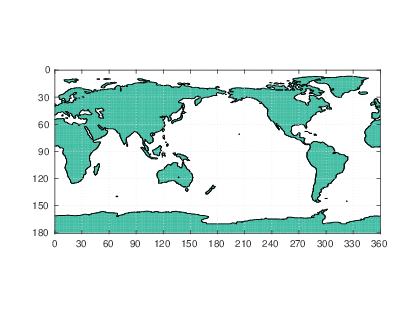}
    \put(0,105) {(a)}
    \put(0,60) {(b)}
    \put(50,0) {\footnotesize $180\lambda/\pi$}
    \put(-3,25) {\footnotesize \rotatebox{90}{$180\theta/\pi$}}
   \end{overpic}
 \end{minipage}
\begin{minipage}{.49\textwidth} 
   \begin{overpic}[width=\textwidth,trim=55 10 68 20,clip]{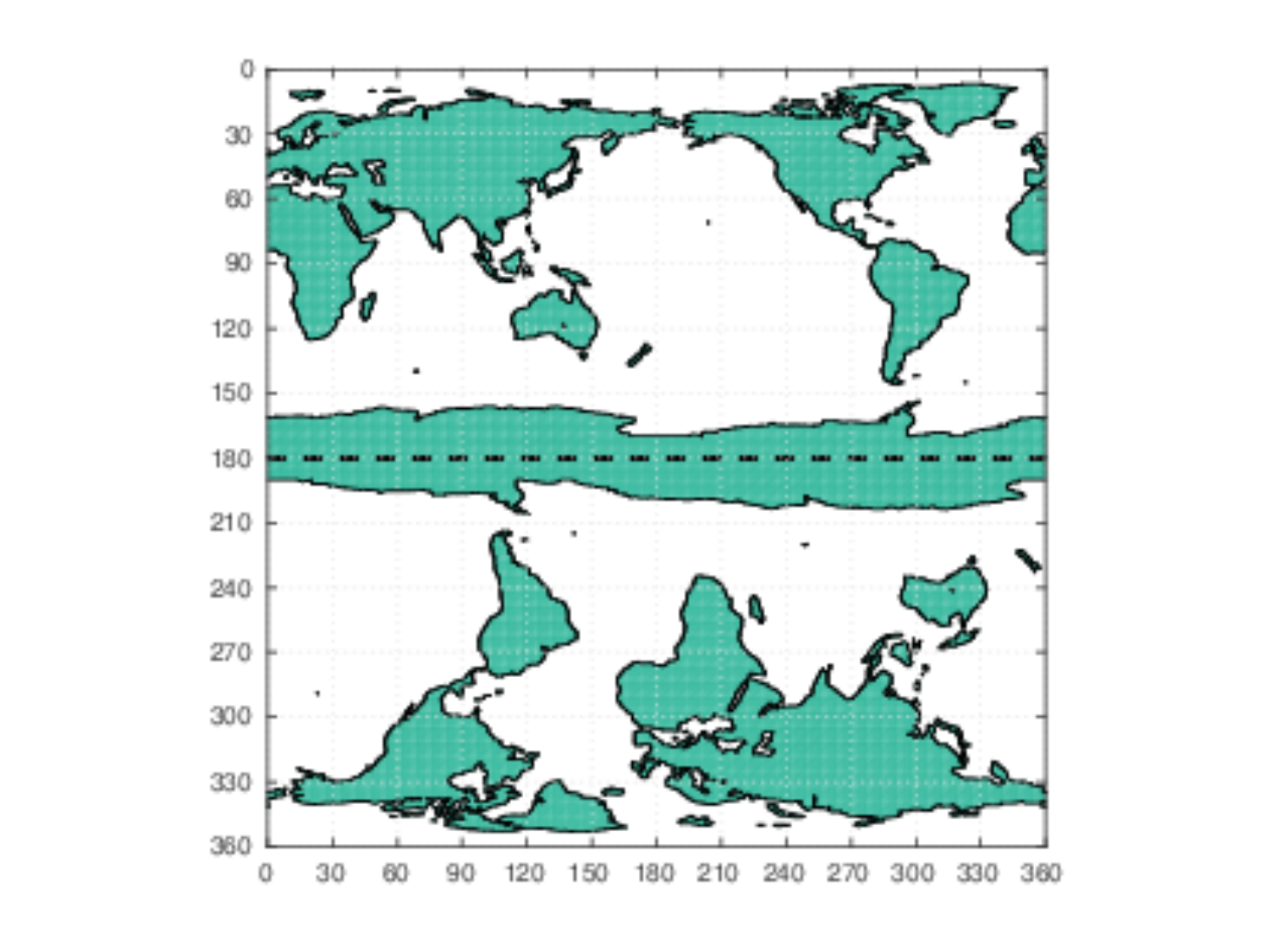}
  \put(0,95) {(c)}
  \put(55,0) {\footnotesize $180\lambda/\pi$}
  \put(0,45) {\footnotesize \rotatebox{90}{$180\theta/\pi$}}
 \end{overpic}
 \end{minipage}
 \caption{The DFS method applied to the globe. (a) The land masses on the surface of earth. (b) The projection of the land masses using latitude-longitude coordinates. (c) Land masses after applying the DFS method. This is a BMC-I ``function'' that is periodic in longitude and latitude.}
 \label{fig:Altas}
\end{figure}
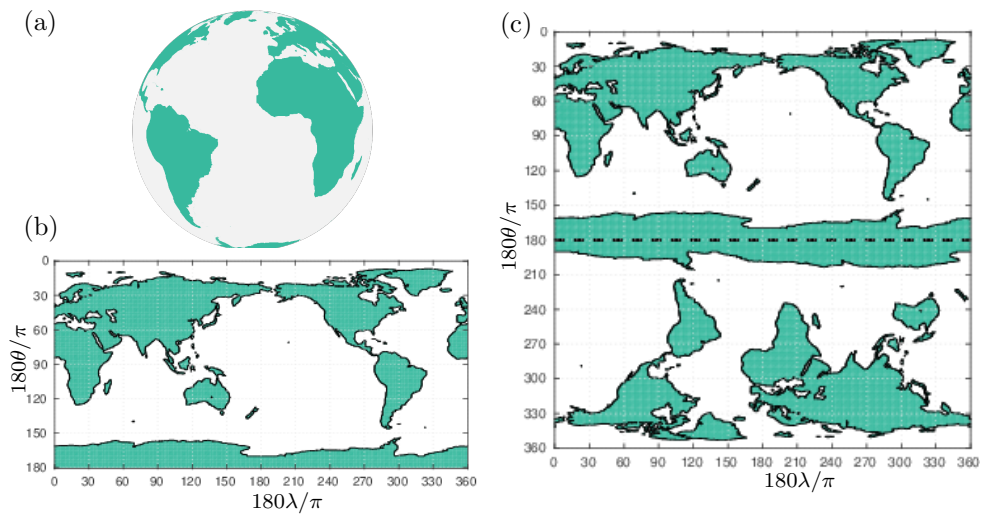

% \subsection{Low rank function approximation} 
% Another, naive approach for representing functions in two variables is a tensor-product construction. 
% \[
% \begin{aligned} 
% %\text{Polar: } &  f(r,\theta) = \sum_{\ell=0}^\infty \sum_{m=0}^\infty c_{\ell,m} T_{\ell}(r) e^{i m\theta},\\
% f(\lambda,\theta) = \sum_{\ell=0}^\infty\sum_{m=-\ell}^\ell c_{\ell,m} T_{\ell}(\frac{(\theta-\pi)}{2\pi}) e^{i m\lambda},\\
% \end{aligned}
% \]
% where $T_{\ell}(x) = \cos(\ell \cos^{-1}x )$ for $-1\leq x\leq 1$ is the degree $\ell$ Chebyshev polynomial (of the first kind). 
% This construction can be intimately connected to fast transforms based on the 
% fast Fourier transform (FFT) and discrete cosine transform (DCT). In this framework all 
% algorithms must be derived from scratch. 
% 
% For many functions that occur in practice one can do better by exploiting a low rank approximation ,  

\section{Low rank approximation for functions on the sphere}\label{sec:lowRankApproximation}
In~\cite{Chebfun2}, low rank techniques 
for numerical computations with bivariate functions was explored. It is now
the technology employed in the two-dimensional side of Chebfun~\cite{Chebfun} with benefits 
that include a compressed representation of functions and efficient algorithms that heavily rely on 
1D technology~\cite{Townsend_15_01}. 
Here, we extend this framework to the approximation of functions on the sphere. 

A function $\tilde{f}(\lambda,\theta)$ is of rank $1$ if it is nonzero and can be written 
as a product of univariate functions, i.e., $\tilde{f}(\lambda,\theta) = c_1(\theta)r_1(\lambda)$. 
A function is of rank at most $K$ if it can be expressed as a sum of $K$ rank $1$ functions.
Here, we describe how to compute rank $K$ approximations of BMC-I functions
that preserve the BMC-I structure. 

\subsection{Structure-preserving Gaussian elimination on functions}\label{sec:StructurePreservingGE}
As an algorithm on $n\times n$ matrices, GE with complete, rook, or 
maximal volume pivoting (but not partial pivoting) is known for its 
rank-revealing properties~\cite{foster2006comparison}. That is, after $K<n$ steps the GE procedure 
can construct a rank~$K$ approximation of a matrix that is close to the best rank $K$ approximation, particularly when 
that matrix comes from sampling a smooth function~\cite{Chebfun2}.  GE for constructing 
low rank approximations is ubiquitous and also goes under the names --- with a variety of pivoting strategies ---
adaptive cross approximation~\cite{bebendorf2000approximation}, two-sided interpolative 
decomposition~\cite{halko2011finding}, and Geddes--Newton approximation~\cite{carvajal2005hybrid}.
 
GE has a natural continuous analogue for functions that immediately 
follows by replacing the matrix $A$ in
the GE step, i.e., $A\leftarrow A - A(:,j)A(i,:)/A(i,j)$, with a function~\cite{Chebfun2}. The 
first step of GE on a BMC function $\tilde{f}$ with 
pivot $(\lambda_*,\theta_*)$ is 
\begin{equation} 
\tilde{f}(\lambda,\theta) \quad \longleftarrow\quad \tilde{f}(\lambda,\theta) - \underbrace{\frac{\tilde{f}(\lambda_*,\theta)\tilde{f}(\lambda,\theta_*)}{\tilde{f}(\lambda_*,\theta_*)}}_{\text{A rank~$1$ approx.~to $\tilde{f}$}}.
\label{eq:GEstep}
\end{equation} 
The GE procedure continues by repeating the same step on the residual. 
That is, the second GE step selects another pivot and 
repeats~\eqref{eq:GEstep}, then $\tilde{f}$ is updated before another GE step
is taken, and so on.  If the pivot locations 
are chosen carefully, then the rank $1$ updates at each step can be accumulated
and after $K$ steps the GE procedure constructs a rank~$K$ 
approximation to the original function $\tilde{f}$, i.e., 
\begin{equation}
 \tilde{f}(\lambda,\theta) \approx \sum_{j=1}^K d_j c_j(\theta)r_j(\lambda),
\label{eq:lowrankRepresentation}
\end{equation} 
where $d_j$ are quantities determined by the pivot values, $c_j$ are the column 
slices, and $r_j$ are the row slices taken during the GE procedure. 

In principle, the GE procedure may continue {\em ad infinitum} as functions can have infinite rank, 
but in practice we terminate the process after a finite number of steps and settle for 
a low rank approximation. Thus, we refer to this as an iterative variant of GE. Two theorems that 
show why smooth functions are typically of low rank can be found in~\cite[Theorems~3.1 \&~3.2]{TownsendThesis}.

Unfortunately, GE on $\tilde{f}$ with any of the standard pivoting strategies 
destroys the BMC structure immediately and the constructed low rank 
approximants are rarely continuous functions on the sphere. We seek a pivoting 
strategy that preserves the BMC structure.  Motivated by the pivoting strategy for 
symmetric indefinite matrices~\cite{bunch1971direct}, 
we employ $2\times 2$ pivots. 
We first consider preserving the BMC structure, before making a small modification to the algorithm for 
BMC-I structure. 

\begin{figure} 
\centering
\includegraphics[width=.36\textwidth,trim=220 410 220 220, clip]{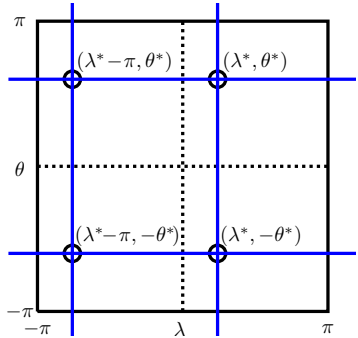}
\caption{A $2\times 2$ GE pivot (black circles) and corresponding rows and columns (blue lines). 
We only select pivots of this form during the GE procedure.
The dotted lines hint at the BMC structure of the function, see~\eqref{eq:depictBMC}.}
\label{fig:pivotLocations}
\end{figure} 

After some deliberation, one concludes that if the 
pivots $(\lambda^*,\theta^*)\in[0,\pi]^2$ and $(\lambda^*-\pi,-\theta^*)$ are 
picked simultaneously, then the GE step does preserve the BMC structure.
Figure~\ref{fig:pivotLocations} shows an example of the $2\times 2$ pivot matrices that
we are considering. To see why such $2\times 2$ pivots 
preserve the BMC structure, let $M$ be the associated $2\times 2$ pivot matrix given by 
\begin{equation}
 M = \begin{bmatrix}\tilde{f}(\lambda^*-\pi,\theta^*) & \tilde{f}(\lambda^*,\theta^*) \cr \tilde{f}(\lambda^*-\pi,-\theta^*) & \tilde{f}(\lambda^*,-\theta^*)\end{bmatrix} = \begin{bmatrix}\tilde{f}(\lambda^*-\pi,\theta^*) & \tilde{f}(\lambda^*,\theta^*) \cr \tilde{f}(\lambda^*,\theta^*) & \tilde{f}(\lambda^*-\pi,\theta^*)\end{bmatrix}, 
\label{eq:pivotMatrix}
\end{equation} 
where the last equality follows from the BMC structure of $\tilde{f}$, see~\eqref{eq:BMCsphere}. The matrix $M$ is a $2\times 2$ 
centrosymmetric matrix and assuming $M$ is an invertible matrix, $M^{-1}$ is also 
centrosymmetric. Therefore, the corresponding GE step on a BMC function $\tilde{f}$ 
takes the following form:
\begin{equation}
 \tilde{f}(\lambda,\theta) \quad \longleftarrow\quad \tilde{f}(\lambda,\theta) - \begin{bmatrix}\tilde{f}(\lambda^*-\pi,\theta) & \tilde{f}(\lambda^*,\theta)\\[3pt] \end{bmatrix} M^{-1} \begin{bmatrix}\tilde{f}(\lambda,\theta^*)\\[3pt] \tilde{f}(\lambda,-\theta^*)\end{bmatrix}.
\label{eq:GEstructurePreservingStep}
\end{equation}
It is a simple matter now to check, using~\eqref{eq:BMCsphere}, that~\eqref{eq:GEstructurePreservingStep} 
preserves the BMC structure of $\tilde{f}$. The key property is that $M^{-1}$ commutes with the exchange matrix because it is centrosymmetric, 
i.e., $JM^{-1} = M^{-1}J$, where $J$ is the matrix formed by swapping 
the rows of the $2\times 2$ identity matrix.

We must go further because the GE step in~\eqref{eq:GEstructurePreservingStep} has two major drawbacks: (1) 
it is not valid unless $M$ is invertible\footnote{Even for mundane functions $M^{-1}$ may not exist. 
For example, if $\tilde{f}\equiv 1$, then 
for all $(\lambda^*,\theta^*)\in[0,\pi]\times [0,\pi]$ the resulting pivot matrix 
$M$ is the matrix of all ones and is singular.}, and (2) it suffers from severe numerical difficulties when $M$ 
is close to singular. To overcome these
failings we replace $M^{-1}$ in~\eqref{eq:GEstructurePreservingStep} by the 
$\epsilon$-pseudoinverse of $M$, denoted by $\Minv$~\cite[Sec.~5.5.2]{Golub_2012_01}. We deliberately 
leave $\epsilon\geq0$ as an algorithmic parameter that we select later. 
\begin{definition}
Let $A$ be a matrix and $\epsilon\geq 0$. If $A = U\Sigma V^*$ is the singular value decomposition of $A$ with
$\Sigma = {\rm diag}(\sigma_1,\ldots,\sigma_n)$ and $\sigma_{k+1} \leq \epsilon < \sigma_k$, then
\[
 A^{\dagger_\epsilon} = V \Sigma^{\dagger_\epsilon}U^*,\qquad \Sigma^{\dagger_\epsilon} = {\rm diag}\left(\sigma_1^{-1}, \ldots, \sigma_k^{-1}, 0, \ldots, 0 \right). 
\]
\label{def:pseudoinverse}
\end{definition}
We discuss the properties of $\Minv$ in the next section, but note here that since $M$ is centrosymmetric so is $\Minv$ (see \eqref{eq:SpectralDecomp}).  Also, the singular values of $M$ are simply
\begin{align}
\sigma_1(M) = \max\{|a+b|,|a-b|\}\;\;\text{and}\;\;\sigma_2(M) = \min\{|a+b|,|a-b|\},
\label{eq:SingularValues}
\end{align}
where $a = \tilde{f}(\lambda^*-\pi,\theta^*)$ and $b=\tilde{f}(\lambda^*,\theta^*)$.
Replacing $M^{-1}$ by $\Minv$ in~\eqref{eq:GEstructurePreservingStep} gives 
\begin{equation}
 \tilde{f}(\lambda,\theta) \quad \longleftarrow\quad \tilde{f}(\lambda,\theta) - \begin{bmatrix}\tilde{f}(\lambda^*-\pi,\theta) & \tilde{f}(\lambda^*,\theta)\\[3pt] \end{bmatrix} \Minv \begin{bmatrix}\tilde{f}(\lambda,\theta^*)\\[3pt] \tilde{f}(\lambda,-\theta^*)\end{bmatrix}.
\label{eq:GEstructurePreservingStepPseudoinverse}
\end{equation} 
If $M$ is well-conditioned, then~\eqref{eq:GEstructurePreservingStepPseudoinverse} 
is the same as~\eqref{eq:GEstructurePreservingStep} because $\Minv = M^{-1}$ 
when $\sigma_2(M)>\epsilon$. However, if $M$ is singular or 
near-singular, then $\Minv$ can be thought of as a surrogate for $M^{-1}$.  
The BMC structure of a function is preserved by~\eqref{eq:GEstructurePreservingStepPseudoinverse} 
since $\Minv$ is still centrosymmetric, for any $\epsilon \geq 0$.
%$\Minv$ is centrosymmetric and commutes  with the exchange matrix.  

Now that~\eqref{eq:GEstructurePreservingStepPseudoinverse} is valid for all nonzero $2\times 2$
pivot matrices, we want to design a strategy to pick ``good'' pivot matrices. 
This allows us to accumulate the GE updates to construct 
low rank approximants to the original function $\tilde{f}$. 
In principle,
we pick $(\lambda^*,\theta^*)\in[0,\pi]\times [0,\pi]$ so that the resulting matrix
$M$ in~\eqref{eq:pivotMatrix} maximizes $\sigma_1(M)$. This is 
the $2\times 2$ pivot analogue of complete pivoting. In practice, 
we settle for a pivot matrix that leads to a large, but not necessarily the maximum 
$\sigma_1(M)$, by searching for $(\lambda^*,\theta^*)$ on a coarse 
discrete grid of $[-\pi,\pi]\times [0,\pi]$. We have found that this pivoting strategy is very effective
for constructing low rank approximants using~\eqref{eq:GEstructurePreservingStepPseudoinverse}.

Unfortunately, the GE procedure does not necessarily preserve the BMC-I structure of a function 
in the sense that the constructed rank $1$ terms in~\eqref{eq:lowrankRepresentation} do not have 
to be constant for $\theta=0$ and $\theta=\pm \pi$ --- 
it is only the complete sum of all the rank $1$ terms that has this property.  If $\tilde{f}$ happens to be zero
along $\theta=0$ and $\theta=\pm \pi$, then each rank $1$ term will have BMC-I structure.  
This suggests that for a BMC-I function one can first ``zero-out'' the function along $\theta=0$ and $\theta=\pm \pi$ and 
then apply the GE procedure to the modified function.  
That is, we first use the rank~$1$ correction: 
\begin{equation}
\tilde{f}(\lambda,\theta) \quad \longleftarrow\quad \tilde{f}(\lambda,\theta) - \tilde{f}(\lambda^{*},\theta),
\label{eq:ZeroOut}
\end{equation}
for some $-\pi \leq \lambda^{*} \leq \pi$.  Afterwards, the GE procedure for preserving 
the BMC structure of a function can be used and the BMC-I structure is automatically preserved. 

Figure~\ref{fig:GaussianElimination} summarizes the GE algorithm that 
preserves the BMC structure of functions and constructs structure-preserving low rank 
approximations. The description given in Figure~\ref{fig:GaussianElimination} is a continuous idealization of the algorithm
that is used in the spherefun constructor as the continuous functions are discretized and  
the GE procedure is terminated after a finite number of steps. Moreover, the 
spherefun constructor only works on the original function $f$ on $[-\pi,\pi]\times [0,\pi]$ 
and mimics the GE procedure on $\tilde{f}$ by using the BMC-I symmetry. This saves a
factor of $2$ in computational cost. 

%when the residual uniformly falls below machine precision relative 
%to the maximum of the original function $\tilde{f}$.
\begin{figure} 
 \centering 
 \fbox{
 \parbox{.97\textwidth}{ 
 \textbf{Algorithm: Structure-preserving GE on BMC functions}
 
 \vspace{.2cm}
 
 \textbf{Input:} A BMC function $\tilde{f}$ and a coupling parameter $0\leq \alpha\leq 1$\\[-7pt] 
 
 \textbf{Output:} A structure-preserving low rank approximation $\tilde{f}_k$ to $\tilde{f}$\\[-3pt]
 
 Set $\tilde{f}_0=0$ and $\tilde{e}_0 = \tilde{f}$.\\[-7pt]
  
 \textbf{for} $k=1,2,3,\ldots,$\\[-10pt]
 
 $\quad$ Find $(\lambda_{k},\theta_{k})$ such that $M =  \begin{bmatrix}a & b \cr b & a\end{bmatrix}$, where $a = \tilde{e}_{k-1}(\lambda_{k-1}-\pi,\theta_{k-1})$ and 
 
 $\quad$ $b = \tilde{e}_{k-1}(\lambda_{k-1},\theta_{k-1})$ has maximal $\sigma_1(M)$ (see~\eqref{eq:SingularValues}).\\[-7pt]
  
 $\quad$ Set $\epsilon = \alpha\sigma_1(M)$. \\[-10pt]
   
 $\quad$ $\tilde{e}_{k} = \tilde{e}_{k-1} - \begin{bmatrix}\tilde{e}_{k-1}(\lambda_k-\pi,\theta) & \tilde{e}_{k-1}(\lambda_k,\theta)\\[3pt] \end{bmatrix} \Minv \begin{bmatrix}\tilde{e}_{k-1}(\lambda,\theta_k)\\[3pt] \tilde{e}_{k-1}(\lambda,-\theta_k)\end{bmatrix}$.\\
 
 $\quad$ $\tilde{f}_{k} = \tilde{f}_{k-1} - \begin{bmatrix}\tilde{e}_{k-1}(\lambda_k-\pi,\theta) & \tilde{e}_{k-1}(\lambda_k,\theta)\\[3pt] \end{bmatrix} \Minv \begin{bmatrix}\tilde{e}_{k-1}(\lambda,\theta_k)\\[3pt] \tilde{e}_{k-1}(\lambda,-\theta_k)\end{bmatrix}$.\\[-7pt]
  
 \textbf{end}
 }}
 \caption{A continuous idealization of our structure-preserving GE procedure on BMC 
 functions. In practice we use a discretization of this procedure and terminate it after a finite 
 number of steps.}
\label{fig:GaussianElimination}
\end{figure}

In practice, to make the spherefun constructor computationally efficient we use the same
algorithmic ideas as in~\cite{Chebfun2}, with the only major difference 
being the pivoting strategy. Phase one of the constructor is designed to estimate the number 
of GE steps and pivot locations required to approximate the BMC-I function $\tilde{f}$. This 
costs $\mathcal{O}(K^3)$ operations, where $K$ is the numerical rank of $\tilde{f}$. Phase two is designed
to resolve the GE column and row slices and costs $\mathcal{O}(K^2(m+n))$ operations, where 
$m$ and $n$ are the number of Fourier modes needed to resolve the columns and rows, respectively. 
The total cost of the spherefun constructor is $\mathcal{O}(K^3+K^2(m+n))$ operations.
For more implementation details on the constructor, we refer the reader to~\cite{Chebfun2}. 

In infinite precision, one may wonder if the GE procedure in Figure~\ref{fig:GaussianElimination} exactly 
recovers a finite rank function.  This is indeed the case. That is, if $\tilde{f}$ is a function of rank $K$ 
in the variables $(\lambda,\theta)$, then the GE procedure terminates after constructing a rank $K$ approximation
and that approximant equals $\tilde{f}$.
\begin{theorem}
 If $\tilde{f}$ is a rank $K$ BMC function on $[-\pi,\pi]^2$, then the GE procedure in Figure~\ref{fig:GaussianElimination} 
 constructs a rank $K$ approximant and $\tilde{f}$ is exactly recovered.  
 \label{thm:BMCrankDrop}
\end{theorem}
\begin{proof}
 Let $(\lambda^*,\theta^*)$ and $(\lambda^*-\pi,-\theta^*)$ be the selected pivot 
 locations in the first GE step and $\Minv$ the corresponding $2\times 2$ pivot matrix. If $\Minv$ 
 is a rank $k$ ($k=1$ or $k=2$) matrix, then GE will form a rank $k$ update in~\eqref{eq:GEstructurePreservingStepPseudoinverse}. 
 Either way, by the generalized Guttman additivity rank formula~\cite[Cor.~19.2]{Matsaglia_74_01}, we have 
 \[
  {\rm rank}( \tilde{f} ) = {\rm rank}( \Minv ) + {\rm rank}\left( \tilde{f} - \begin{bmatrix}\tilde{f}(\lambda^*-\pi,\cdot) & \tilde{f}(\lambda^*,\cdot)\\[3pt] \end{bmatrix} \Minv \begin{bmatrix}\tilde{f}(\cdot,\theta^*)\\[3pt] \tilde{f}(\cdot,-\theta^*)\end{bmatrix} \right), 
 \]
where ${\rm rank}(\cdot)$ denotes the rank of the function or matrix. If 
the GE procedure constructs a rank $k$ ($k=1$ or $k=2)$ approximation in the first 
step, then the rank of the residual is ${\rm rank}(\tilde{f})-k$. Repeating this until 
the residual is of rank $0$, shows that the rank of $\tilde{f}$ and the 
final approximant are the same. The exact recovery result 
follows because the only function of rank $0$ is the zero function so the final 
residual is zero. 
\end{proof}

A band-limited function on the sphere is one that can be expressed as a finite sum of spherical harmonics, similar to a band-limited function on the interval being expressed as a finite Fourier series.  Since each spherical harmonic function is itself a rank 1 function~\cite[Sec.~14.30]{NIST}, Theorem~\ref{thm:BMCrankDrop} also implies that our GE procedure exactly recovers  band-limited functions after a finite number of steps. 

For infinite rank functions, the GE procedure 
in Figure~\ref{fig:GaussianElimination} requires in principle an infinite number of 
steps. We can prove that the successive low rank approximants constructed by GE
converge to $\tilde{f}$ under certain conditions on $\tilde{f}$ (see Section~\ref{sec:GEconvergence}).  Thus, the procedure can be 
terminated after a finite, often small, number of steps, giving an accurate low-rank approximant.  
In the spherefun constructor, we terminate the procedure when the residual falls below machine precision relative 
to an estimate of the absolute maximum of the original function.

If the parameter $\epsilon \geq 0$ for determining $\Minv$ is 
too large, then severely ill-conditioned pivot matrices are allowed and the 
algorithm suffers from a loss of accuracy. If $\epsilon$ is too small, then $\Minv$ is 
almost always of rank $1$ and Theorem~\ref{thm:BMCrankDrop}
shows that the progress of GE is hindered. 
We choose $\epsilon$ to be $\epsilon = \alpha\sigma_1(M)$, where $\alpha = 1/100$. 
In other words, we use $M^{-1}$ in the GE step if $\sigma_1(M)/\sigma_2(M)<100$ and $\Minv$ otherwise. 
We call $\alpha$ the {\em coupling parameter} for reasons that are explained
in Section~\ref{sec:GEinterpretation}. 

Figure~\ref{fig:SingularityAtPole} shows the importance of constructing approximants
that preserve the BMC-I structure of functions on the sphere since an artificial pole 
singularity is introduced in each rank $1$ term when the structure is not, 
reducing the accuracy for derivatives. A close inspection of subplots (e) and (f) reveals pole singularities. 

\begin{figure} 
\begin{center}
\begin{minipage}{0.2\textwidth}
\begin{overpic}[width=\textwidth]{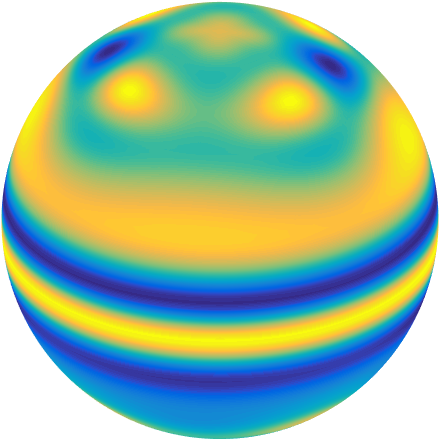} 
 \put(0,85) {(a)}
 \end{overpic}
 \end{minipage}
\begin{minipage}{0.2\textwidth}
\begin{overpic}[width=\textwidth]{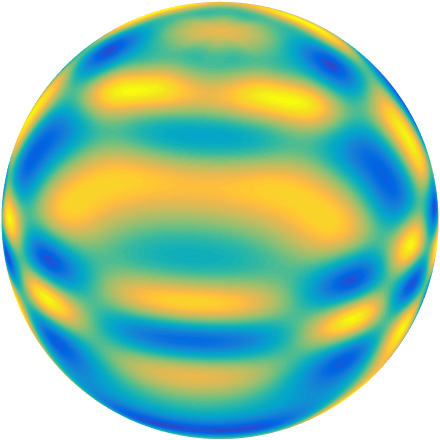} 
 \put(0,85) {(b)}
 \end{overpic}
 \end{minipage}
\begin{minipage}{0.2\textwidth}
\begin{overpic}[width=\textwidth]{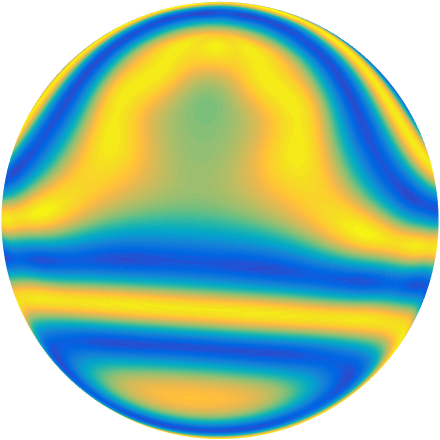} 
 \put(0,85) {(c)}
 \end{overpic}
 \end{minipage}
\begin{minipage}{0.2\textwidth}
\begin{overpic}[width=\textwidth]{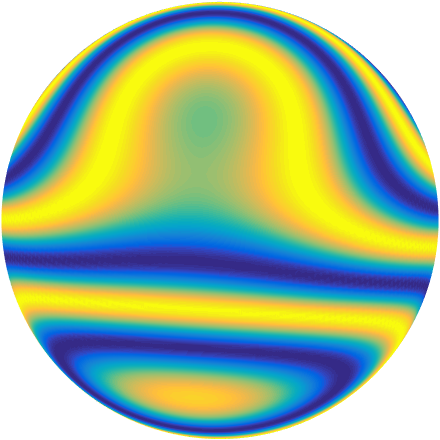} 
 \put(0,85) {(d)}
 \end{overpic}
 \end{minipage}
 
 \vspace{.3cm}
 
\begin{minipage}{0.2\textwidth}
\begin{overpic}[width=\textwidth]{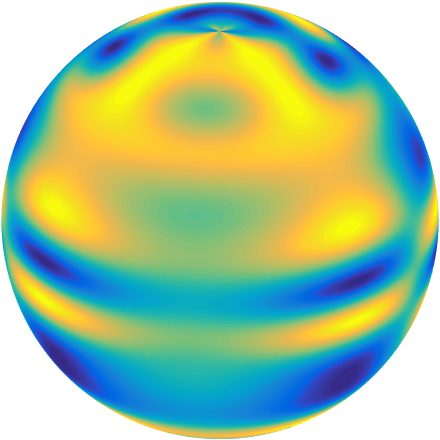} 
 \put(0,85) {(e)}
 \end{overpic}
 \end{minipage}
\begin{minipage}{0.2\textwidth}
\begin{overpic}[width=\textwidth]{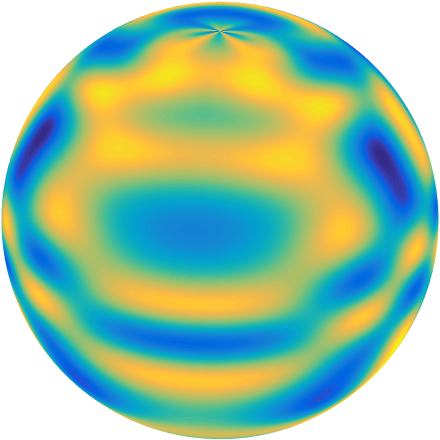} 
 \put(0,85) {(f)}
 \end{overpic}
 \end{minipage}
\begin{minipage}{0.2\textwidth}
\begin{overpic}[width=\textwidth]{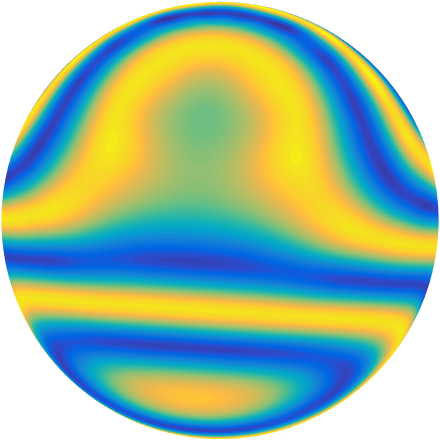} 
 \put(0,85) {(g)}
 \end{overpic}
 \end{minipage}
\begin{minipage}{0.2\textwidth}
\begin{overpic}[width=\textwidth]{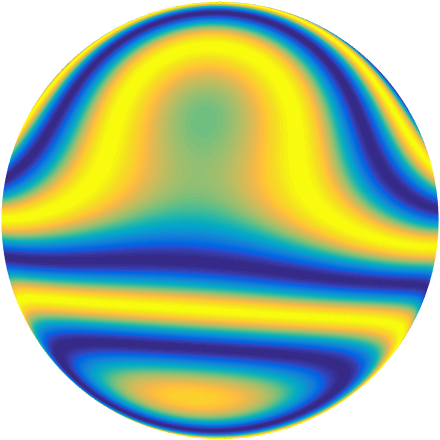} 
 \put(0,85) {(h)}
 \end{overpic}
 \end{minipage}
 \end{center}
\caption{Low rank approximants to $\tilde{f}$ in~\eqref{eq:sphereTestFunc}. 
Figures (a)-(d) are the respective rank 2, 4, 8, and 16 approximants to $\tilde{f}$ constructed by the structure-preserving GE procedure 
in Section~\ref{sec:StructurePreservingGE}. Figures (e)-(h) are the respective rank 2, 4, 8, and 16 approximants to $\tilde{f}$ constructed by the GE procedure 
that is not designed to preserve the BMC-I structure~\cite{Chebfun2}. In figures (e) and (f) one can see that a pole singularity is introduced when 
structure is not preserved.}
\label{fig:SingularityAtPole}
\end{figure}

Our GE procedure samples $\tilde{f}$ along a sparse collection of lines, known as a {\em skeleton}~\cite{Chebfun2}, 
to construct a low rank approximation. This can be seen from the GE step in~\eqref{eq:GEstructurePreservingStepPseudoinverse}, which
only requires 1D slices of $\tilde{f}$. Thus, 
$\tilde{f}$ is sampled on a grid that is not clustered near the pole of the sphere, 
unless it has properties that requires this. Instead, the sample points used for approximating 
$\tilde{f}$ are determined adaptively by the GE procedure and are 
composed of a criss-cross of 1D uniform grids. This means that we can take 
advantage of the low rank structure of functions, while still employing fast FFT-based algorithms.  Figure~\ref{fig:PoleResolution} 
shows the skeleton selected by the GE procedure when constructing a rank $17$ 
approximant of the function $f(x,y,z) = \cos(xz-\sin y)$.

\subsection{Another interpretation of our Gaussian elimination procedure}\label{sec:GEinterpretation}
The GE procedure in Figure~\ref{fig:GaussianElimination}
that employs $2\times 2$ pivots can also be interpreted as two coupled
GE procedures, with a coupling strength of $0\leq\alpha\leq 1$. 
This interpretation connects our method for approximating functions on the sphere to existing approximation techniques 
involving even-odd modal decompositions~\cite{Yee1980}. 
%It also will allow us to derive the convergence properties of the GE procedure (see Section~\ref{sec:GEconvergence}).  

Let $\tilde{f}$ be a BMC function and $M$ be the first $2\times 2$ pivot matrix defined in \eqref{eq:pivotMatrix} and written as
$M = \begin{bmatrix} a & b \\ b & a \end{bmatrix}$.  A straightforward derivation shows $\Minv$ can be expressed as
% This is not correct:
% the singular value decomposition of $\Minv$ can be calculated explicitly as
%\begin{equation}
% \Minv = U \begin{bmatrix} |a + b| \cr & |a - b| \end{bmatrix} U^*, \qquad U = \frac{1}{\sqrt{2}}\begin{bmatrix}1 & 1 \cr 1 & -1 \end{bmatrix}.
%\label{eq:explicitSVD}
%\end{equation}
\begin{align}
\Minv = \begin{bmatrix} \frac{1}{\sqrt{2}} & \frac{1}{\sqrt{2}} \\ \frac{1}{\sqrt{2}} & -\frac{1}{\sqrt{2}} \end{bmatrix}
\begin{bmatrix} 
\meven & \phantom{\frac{1}{\sqrt{2}}} \\
\phantom{\frac{1}{\sqrt{2}}} & \modd 
\end{bmatrix}
\begin{bmatrix} \frac{1}{\sqrt{2}} & \frac{1}{\sqrt{2}} \\ \frac{1}{\sqrt{2}} & -\frac{1}{\sqrt{2}} \end{bmatrix},
\label{eq:SpectralDecomp}
\end{align}
where the possible values of $\meven$ and $\modd$ are given by
\begin{align}
(\meven,\modd) =
\begin{cases}
(1/(a+b),0), & \text{if $|a-b|<\alpha |a+b|$}, \\
(0,1/(a-b)), & \text{if $|a+b|<\alpha |a-b|$}, \\
(1/(a+b),1/(a-b)), & \text{otherwise}. \\
\end{cases}
\label{eq:mvalues}
\end{align}
Recall from the previous section that $\alpha = \epsilon/\sigma_1(M)=\epsilon/\max\{|a+b|,|a-b|\}$ and we set $\alpha=1/100$.   
If neither the first or second cases applies in~\eqref{eq:mvalues}, then $\Minv=M^{-1}$.
%Equation \eqref{eq:SpectralDecomp} is just the spectral decomposition of $\Minv$.

\begin{figure}[t!]
\begin{center}
 \begin{minipage}{.4\textwidth}
  \includegraphics[width=0.7\textwidth]{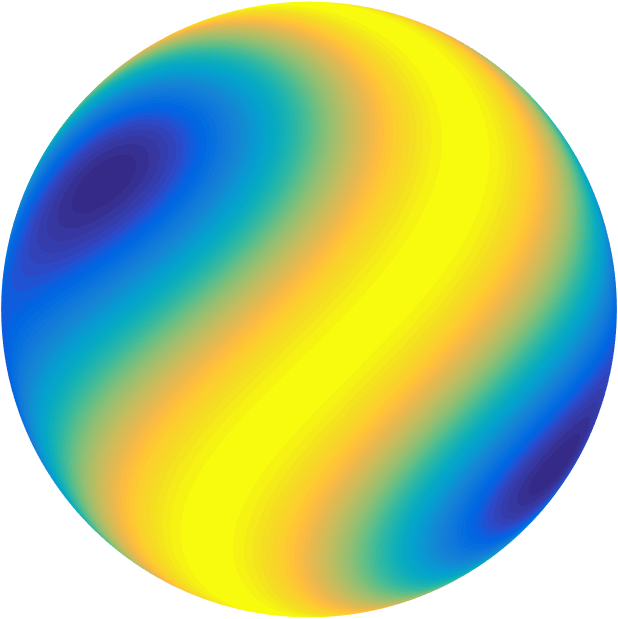}
 \end{minipage}
 \begin{minipage}{.4\textwidth}
  \includegraphics[width=0.7\textwidth]{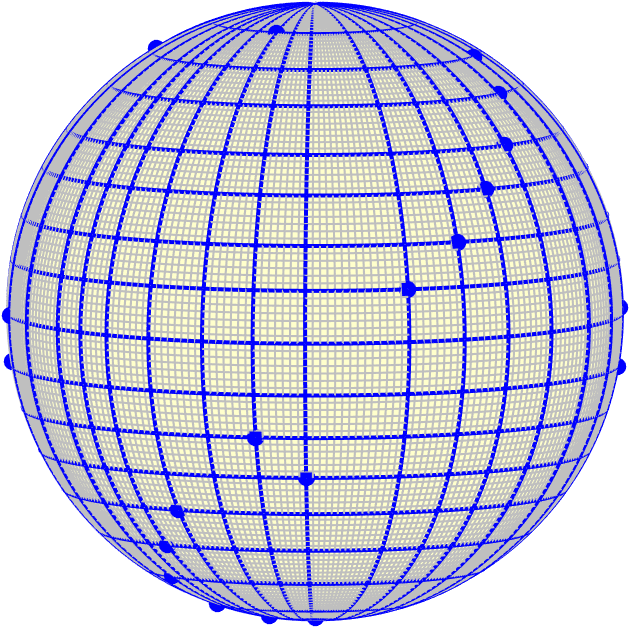}
 \end{minipage}
 \end{center}
 \caption{Left: The function $f(x,y,z)=\cos(xz-\sin y)$ on the unit sphere.  Right: The ``skeleton'' used to approximate $f$ by \texttt{spherefun}.
 The blue dots are the entries of the $2\times 2$ pivot matrices used by GE. The GE procedure 
 only samples $f$ along the blue lines. The underlying tensor grid (in gray) shows 
 the sampling grid required without low rank techniques, which cluster near the poles.}
 \label{fig:PoleResolution} 
\end{figure}

We can use \eqref{eq:SpectralDecomp} to re-write the GE step in~\eqref{eq:GEstructurePreservingStepPseudoinverse} as
\begin{equation}
\begin{aligned} 
  \tilde{f}(\lambda,\theta) \quad \longleftarrow\quad \tilde{f}(\lambda,\theta)& -\frac{\meven}{2}\left(\tilde{f}(\lambda^*-\pi,\theta) + \tilde{f}(\lambda^*,\theta)\right)\left(\tilde{f}(\lambda,\theta^*)+\tilde{f}(\lambda,-\theta^*)\right)\\
  &-\frac{\modd}{2}\left(\tilde{f}(\lambda^*-\pi,\theta) - \tilde{f}(\lambda^*,\theta)\right)\left(\tilde{f}(\lambda,\theta^*)-\tilde{f}(\lambda,-\theta^*)\right).
\end{aligned}
\label{eq:FirstGEstepSplit}
\end{equation}
Now we make a key observation.  Let $\feven = g + h$ and $\fodd = g-h$, where $g$ and $h$ are defined in~\eqref{eq:BMCsphere}, and note that we can decompose $\tilde{f}$ into a sum of two BMC functions:
\begin{align}
\tilde{f} = 
\frac12
\underbrace{
\begin{bmatrix}
\feven & \feven \\
{\tt flip}(\feven) & {\tt flip}(\feven)
\end{bmatrix}}_{ =\fteven}
+
\frac12
\underbrace{
\begin{bmatrix}
\fodd & -\fodd \\
-{\tt flip}(\fodd) & {\tt flip}(\fodd)
\end{bmatrix}}_{ =\ftodd}.
\label{eq:EvenOddSplit}
\end{align}
Using the definitions of $\fteven$ and $\ftodd$ in this decomposition,  the GE step in \eqref{eq:FirstGEstepSplit} can then be written as
\begin{equation}
  \tilde{f}(\lambda,\theta) \; \longleftarrow\;  \frac12(\fteven(\lambda,\theta) - \meven\fteven(\lambda^*,\theta)\fteven(\lambda,\theta^*)) + \frac12(\ftodd(\lambda,\theta) - \modd\ftodd(\lambda^*,\theta)\ftodd(\lambda,\theta^*)),
\label{eq:GEstepSplit}
\end{equation}
%\begin{equation}
%\begin{aligned} 
%  \tilde{f}(\lambda,\theta) \quad \longleftarrow\quad & \frac12(\fteven(\lambda,\theta) - \meven\fteven(\lambda^*,\theta)\fteven(\lambda,\theta^*))\;\; + \\ 
%  & \frac12(\ftodd(\lambda,\theta) - \modd\ftodd(\lambda^*,\theta)\ftodd(\lambda,\theta^*)),
%\end{aligned}
%\label{eq:GEstepSplit}
%\end{equation}
which hints that the step is equivalent to two coupled GE steps on the functions $\fteven$ and $\ftodd$.  This connection can be made complete by noting that $a+b=\fteven(\lambda^*,\theta^*)$ and $a-b=\ftodd(\lambda^*,\theta^*)$ in the definition of $\meven$ and $\modd$ in \eqref{eq:mvalues}.  

The coupling of the two GE steps in~\eqref{eq:GEstepSplit} is through the parameter $\alpha$ used to define $\Minv$.   
From \eqref{eq:mvalues} we see that when $|\ftodd(\lambda^*,\theta^*)| < \alpha |\fteven(\lambda^*,\theta^*)|$, the GE step 
applies only to $\fteven$.  Since we select the pivot matrix $M$ so that $\sigma_1(M)$ (see Figure \ref{fig:GaussianElimination}) 
is maximal, this step corresponds to GE with complete pivoting on $\fteven$ and does not alter $\ftodd$.  
Similarly, a GE step with complete pivoting is done on $\ftodd$ when $|\fteven(\lambda^*,\theta^*)| < \alpha |\ftodd(\lambda^*,\theta^*)|$.  
If neither of these conditions is met, then \eqref{eq:GEstepSplit} corresponds to an interesting mix between 
a GE step on $\fteven$ (or $\ftodd$) with complete pivoting and another GE step on $\ftodd$ (or $\fteven$) with a nonstandard pivoting strategy.

 If one takes $\alpha = 1$, then the GE steps in~\eqref{eq:GEstepSplit} on $\fteven$ and $\ftodd$ are fully 
 decoupled. In this regime the algorithm in Section~\ref{sec:StructurePreservingGE} 
 is equivalent to applying GE with complete pivoting to $\fteven$ and 
 $\ftodd$ independently. There is a fundamental issue with this. The 
 rank $1$ terms attained from applying GE to $\fteven$ and $\ftodd$ can not be properly 
 ordered when constructing a low-rank approximant of $\tilde{f}$. 
 %This requires checking for the maximum value of $\fteven$ and $\ftodd$ at each step, but then only using 
 %one of these values, resulting in only a rank 1 update\footnote{Sorting by the magnitude of the pivots does not 
 %recover the proper ordering of the elimination steps~\cite{Chebfun2} (correct citation?)}.  
 By selecting $\alpha<1$, the GE steps are coupled and the rank $1$ terms are (partially) ordered. This also means that 
 a GE step can achieve a rank $2$ update, which reduces the number of pivot searches and improves the 
 overall efficiency of the spherefun constructor.

The decomposition $\tilde{f} = \frac12\feven + \frac12\fodd$ is also important for identifying symmetries that a BMC function obtained 
from the DFS method must possess.   From~\eqref{eq:EvenOddSplit} we see that the function 
$\fteven$ is an even function in $\theta$ and $\pi$-periodic in $\lambda$, while $\ftodd$ is an odd function in $\theta$ and 
$\pi$-antiperiodic\footnote{A function $f(x)$ is $\pi$-antiperiodic if $f(x+\pi) = -f(x)$ for $x\in\mathbb{R}$.} in $\lambda$.   
This even-periodic/odd-antiperiodic decomposition of a BMC function has been used in various guises in the DFS method as 
detailed in~\cite{Yee1980}.  However, in these studies the representations of $\tilde{f}$ were constructed in a purely modal fashion, where 
the symmetries were enforced directly on the 2D Fourier coefficients of $\tilde{f}$.  The representation of $\tilde{f}$ in~\eqref{eq:EvenOddSplit} 
shows how to enforce these symmetries in a purely nodal fashion, i.e., on the values of the function, which appears to be a new observation. 
Our GE procedure produces a low rank approximation to $\tilde{f}$ that preserves these even-periodic and odd-antiperiodic symmetries.% as 
%indicated in \eqref{eq:GEstepSplit}.

We conclude this section by noting another important result of the decomposition of the pseudoinverse of the $2\times2$ pivot 
matrices in~\eqref{eq:SpectralDecomp}.  Applying this decomposition to each pivot matrix after the GE procedure in 
Figure~\ref{fig:GaussianElimination} terminates, allows us to write the low rank function that results from the 
algorithm in the form of \eqref{eq:lowrankRepresentation}, with $d_j$ given by the eigenvalues of the  pseudoinverse of 
the pivot matrices.  Furthermore, using \eqref{eq:GEstepSplit}, we can split the approximation as
\begin{align}
 \tilde{f}(\lambda,\theta) \approx \sum_{j=1}^K d_j c_j(\theta)r_j(\lambda) =  \sum_{j=1}^{\Keven} \deven_j \ceven_j(\theta)\reven_j(\lambda) + \sum_{j=1}^{\Kodd} \dodd_j \codd_j(\theta)\rodd_j(\lambda),
\label{eq:lowrankRepresentationSplit}
\end{align} 
where $\Keven + \Kodd=K$. Here, the functions $\ceven_j(\theta)$ and $\reven_j(\lambda)$ for $1\leq j\leq K$ are even and $\pi$-periodic, while 
$\codd_j(\theta)$ and $\rodd_j(\lambda)$ for $1\leq j\leq K$ are odd and $\pi$-antiperiodic.  If $\tilde{f}$ is non-zero at the poles and \eqref{eq:ZeroOut} is employed in the first step of the GE procedure, then $\ceven_1(\theta) = \tilde{f}(\lambda^{*},\theta)$, $\reven_1(\lambda)=1$, and $\deven_1=1$.  The two summations after the last equal sign in \eqref{eq:lowrankRepresentationSplit} provide low rank approximations to $\fteven$ and $\ftodd$, respectively.  The BMC-I structure of the approximation \eqref{eq:lowrankRepresentationSplit} then becomes obvious.

\subsection{Analyzing the structure-preserving Gaussian elimination procedure}\label{sec:GEconvergence}
GE on matrices with partial pivoting is known to be theoretically
unstable in the worst case because each step can increase the absolute magnitude of 
the matrix entries by a factor of $2$. Even though in 
practice this instability is extraordinarily rare, for a convergence 
theorem we need to control the worst-case behavior. 

The so-called {\em growth factor} quantifies the worst possible increase in the absolute maximum 
after a rank one update. The following theorem gives a bound on the growth factor for our structure-preserving 
GE procedure in Figure~\ref{fig:GaussianElimination}.
\begin{lemma} 
 The growth factor for the structure-preserving GE procedure in Figure~\ref{fig:GaussianElimination} 
 is $\leq\max(3,\sqrt{1+4/\alpha})$, where $\alpha$ is the coupling parameter. 
\end{lemma}
\begin{proof} 
It is sufficient to examine the growth factor of the first GE step.  
Let $M$ be the first $2\times 2$ pivot matrix so that $M$ is 
the matrix of the form in~\eqref{eq:pivotMatrix} that maximizes 
$\sigma_1(M)$. Since $M$ maximizes $\sigma_1(M)$ and using \eqref{eq:SingularValues}, 
we have $\sigma_1(M)\geq \|\tilde{f}\|_\infty$, where $\|\tilde{f}\|_\infty$ denotes the absolute maximum of $\tilde{f}$ on $[-\pi,\pi]^2$. 
There are two cases to consider.

\underline{Case $1$: $\sigma_2(M)<\alpha\sigma_1(M)$.} Here $\Minv$ in~\eqref{eq:GEstructurePreservingStepPseudoinverse} with 
$\epsilon = \alpha\sigma_1(M)$ is of rank $1$ and the explicit formula for the spectral 
decomposition of $\Minv$ in~\eqref{eq:SpectralDecomp} shows that 
\[
 \left\| \tilde{f} - \begin{bmatrix}\tilde{f}(\lambda^*-\pi,\cdot) & \tilde{f}(\lambda^*,\cdot)\\[3pt] \end{bmatrix} \Minv \begin{bmatrix}\tilde{f}(\cdot,\theta^*)\\[3pt] \tilde{f}(\cdot,-\theta^*)\end{bmatrix}\right\|_\infty 
 \leq \|\tilde{f}\|_\infty + \frac{2\|\tilde{f}\|_\infty^2}{\sigma_1(M)}\leq 3 \|\tilde{f}\|_\infty,
\]
where in the last equality we used $\sigma_1(M)\geq \|\tilde{f}\|_\infty$. Thus, the growth 
factor here is $\leq 3$.

\underline{Case $2$: $\sigma_2(M)\geq\alpha\sigma_1(M)$.} Here $\Minv=M^{-1}$ and 
\[
 \|M^{-1}\|_{\max} \leq \frac{\|\tilde{f}\|_\infty}{{\rm det}(M)} = \frac{\|\tilde{f}\|_\infty}{\sigma_1(M)\sigma_2(M)} \leq \frac{\|\tilde{f}\|_\infty}{\alpha\sigma_1(M)^2}\leq \frac{1}{\alpha\|\tilde{f}\|_\infty},
\]
where $\|M^{-1}\|_{\max}$ denotes the maximum absolute entry of $M^{-1}$. Therefore, we have
\[
 \left\| \tilde{f} - \begin{bmatrix}\tilde{f}(\lambda^*-\pi,\cdot) & \tilde{f}(\lambda^*,\cdot)\\[3pt] \end{bmatrix} \Minv \begin{bmatrix}\tilde{f}(\cdot,\theta^*)\\[3pt] \tilde{f}(\cdot,-\theta^*)\end{bmatrix}\right\|_\infty 
 \leq \|\tilde{f}\|_\infty + \frac{4\|\tilde{f}\|_\infty}{\alpha}\leq \left(1 + \frac{4}{\alpha}\right) \|\tilde{f}\|_\infty.
\]
Thus, the growth factor here is $\leq\sqrt{1+4/\alpha}$ because the GE 
update is of rank $2$.
\end{proof}

Bounding the growth factor leads to a GE convergence result for continuous functions $\tilde{f}(\lambda,\theta)$ 
that satisfy the following property: for each fixed $\theta\in[-\pi,\pi]$, $\tilde{f}(\cdot,\theta)$ is an analytic function in a sufficiently large neighborhood of the complex plane
containing $[-\pi,\pi]$. In approximation theory it is common to 
consider the neighborhood known as the {\em stadium} of radius 
$\beta>0$, denoted by $S_\beta$. 

\begin{definition} 
Let $S_\beta$ with $\beta>0$ be the ``stadium'' of radius $\beta$ in the complex plane consisting of all numbers lying at a distance $\leq \beta$ from 
an interval $[a,b]$, i.e., 
\[
S_\beta = \left\{z\in\mathbb{C} : \inf_{x\in[a,b]} |x - t| \leq \beta \right\}.
\]
\end{definition} 

In the statement of the following theorem the roles of $\lambda$ and $\theta$ can be 
exchanged. 
\begin{theorem} 
 Let $\tilde{f}:[-\pi,\pi]^2\rightarrow\mathbb{C}$ be a 
 BMC function such that $\tilde{f}(\lambda,\cdot)$ is continuous for any $\lambda\in[-\pi,\pi]$ 
 and $\tilde{f}(\cdot,\theta)$ is analytic and uniformly bounded 
 in a stadium $S_\beta$ of radius $\beta=\max(3,\sqrt{1+4/\alpha}) \rho \pi$, $\rho>1$, for any $\theta\in[-\pi,\pi]$. 
 Then, the error after $k$ GE steps decays to zero as $k\rightarrow \infty$, i.e.,
 \[
  \|\tilde{e}_k\|_\infty\rightarrow 0,\qquad k\rightarrow \infty.
 \]
 That is, the sequence of approximants constructed by the structure-preserving 
 GE procedure for $\tilde{f}$ 
 in Figure~\ref{fig:GaussianElimination} converges uniformly to $\tilde{f}$.
 \label{thm:GEconvergence}
\end{theorem}
\begin{proof} 
Let $\tilde{e}_k$ be the error after $k$ GE steps in Figure~\ref{fig:GaussianElimination}. 
Since $\tilde{e}_k$ is a BMC function for $k\geq 0$, $\tilde{e}_k$ can be decomposed into the sum of 
an even-periodic and odd-antiperiodic function, i.e., $\tilde{e}_k = \eteven_k + \etodd_k$ for $k\geq 0$, as discussed in Section~\ref{sec:GEinterpretation}. 
Additionally,  from Section \ref{sec:GEinterpretation}, we know that the structure-preserving GE procedure 
in Figure~\ref{fig:GaussianElimination} can regarded as two coupled GE procedure on the even-periodic 
and odd-antiperiodic parts; see~\eqref{eq:GEstepSplit}.

Thus, we examine the size of $\|\eteven_k\|_\infty$ and $\|\etodd_k\|_\infty$, hoping 
to show that $\|\eteven_k\|_\infty\rightarrow0$ and $\|\etodd_k\|_\infty\rightarrow0$ as
$k\rightarrow\infty$. First, note 
that we can write $k = \keven + \kodd + k^0$, where 
 \[
  \begin{aligned} 
   \keven& = \text{ the number of GE steps in which only $\eteven_k$ is updated},\\
   \kodd& = \text{ the number of GE steps in which only $\etodd_k$ is updated},\\
   k^{0}& = \text{ the number of GE steps in which both $\eteven_k$ and $\etodd_k$ are updated}.
  \end{aligned}
 \]
Since $\eteven_0$ and $\etodd_0$ are continuous, the 
growth factor of GE at each step is $\leq\max(3,\sqrt{1+4/\alpha})$, and $\tilde{e}^{+}_0(\cdot,\theta)$
and $\etodd_0(\cdot,\theta)$ are analytic and uniformly bounded 
in the stadium $S_\beta$ of radius $\beta=  \max(3,\sqrt{1+4/\alpha})\rho \pi$, $\rho>1$, for any $\theta\in[-\pi,\pi]$. 
We know from Theorem $8.2$ in~\cite{Townsend_15_01}, which proves the convergence of 
GE on functions, that
\[
\begin{aligned}
  &\|\eteven_k\|_\infty\rightarrow 0, \qquad \text{ if } \quad \keven+k^{0}\rightarrow \infty, \\
  &\|\etodd_k\|_\infty\rightarrow 0, \qquad \text{ if } \quad \kodd+k^{0}\rightarrow \infty. \\
\end{aligned}
\]
Since either $\keven+k^{0}\rightarrow \infty$ or $\kodd+k^{0}\rightarrow \infty$ as $k\rightarrow\infty$, 
either $\|\eteven_k\|_\infty\rightarrow 0$ or $\|\etodd_k\|_\infty\rightarrow 0$. 
 
We now set out to show that both $\|\eteven_k\|_\infty\rightarrow 0$ and $\|\etodd_k\|_\infty\rightarrow 0$ as $k\rightarrow\infty$, and 
we proceed by contradiction. Suppose that $\|\eteven_k\|_\infty > \delta >0$ for all $k\geq 0$ and hence, the 
number of steps that updated the even-periodic part is finite. Let step $K$ be the last GE step that updated the even-periodic 
part. Now pick $K^*>K$ sufficiently large so 
that $\|\etodd_{K^*}\|_\infty < \delta$ and note that 
the $K^*+1>K$ GE step must update the even-periodic part, contradicting that $K$ was the last step to 
update the even-periodic part. 
We conclude that $\|\tilde{e}_k\|_\infty \leq \|\eteven_k + \etodd_k\|_\infty \leq \|\eteven_k\|_\infty  + \|\eteven_k\|_\infty\rightarrow 0$ as $k\rightarrow \infty$. 
\end{proof}

We expect that one can show that the GE procedure constructs a sequence of low rank approximants
that converges geometrically to $\tilde{f}$, under the same assumptions of Theorem~\ref{thm:GEconvergence}. Moveover,
weaker analyticity assumptions may be possible because~\cite{townsend2016gaussian} could lead to a tighter bound on 
the growth factor of our GE procedure.  In practice, 
even for functions that are a few times differentiable, the low rank approximants constructed by GE
converge to $\tilde{f}$, but our theorem does not prove this. 

\section{A collection of algorithms for numerical computations with function on the sphere}\label{sec:spherefun} 
Low rank approximants have a convenient representation for 
integrating, differentiating, evaluating, and performing many other computational tasks.  We now discuss
several of these operations, which are all available in spherefun.  In the discussion, we assume
that $f$ is a smooth function on the sphere and has been extended to a BMC function, $\tilde{f}$, using 
the DFS method. Then, we suppose that the GE procedure in Figure~\ref{fig:GaussianElimination} has 
constructed a low rank approximation of $\tilde{f}$ as in \eqref{eq:lowrankRepresentation}. 
The functions $c_j(\theta)$ and $r_j(\theta)$ in \eqref{eq:lowrankRepresentation} are $2\pi$-periodic and we represent them with Fourier expansions, i.e., 
\begin{align}
c_j(\theta) = \sum_{k=-m/2}^{m/2-1} a_k^{j} e^{i k \theta},\qquad r_j(\lambda) = \sum_{k=-n/2}^{n/2-1} b_k^{j} e^{i k \lambda},
\label{eq:FourierExpansions}
\end{align}
where $m$ and $n$ are even integers.  We could go further and split the functions $c_j$ and $r_j$ into the functions $\ceven_j$, $\reven_j$,  $\codd$, and $\rodd$ in \eqref{eq:lowrankRepresentationSplit}.  In this case the Fourier coefficients of these functions would satisfy certain properties related to even/odd and $\pi$-periodic/$\pi$-antiperiodic symmetries~\cite{TrigFun}.

In principle, the number of Fourier modes for $c_j(\theta)$ and $r_j(\lambda)$ in~\eqref{eq:FourierExpansions} 
could depend on $j$. Here, we use the same number of modes, $m$, for each $c_j(\theta)$ and, $n$, for each
$r_j(\lambda)$. This allows operations to be more efficient 
as the code can be vectorized.  

\subsubsection*{Pointwise evaluation}
The evaluation of $f(x,y,z)$ on the surface of the sphere, i.e., when $x^2 + y^2 + z^2=1$, 
is computationally very efficient. In fact this immediately follows from the low rank representation for 
$\tilde{f}$ since
\[
 f(x,y,z) = \tilde{f}(\lambda,\theta) \approx \sum_{j=1}^K d_jc_j(\theta)r_j(\lambda), 
\]
where $\lambda = \tan^{-1}(y/x)$ and $\theta = \cos^{-1}(z/(x^2+y^2)^{1/2})$. Thus, 
$f(x,y,z)$ can be calculated by evaluating $2K$ 1D Fourier expansions~\eqref{eq:FourierExpansions} using 
Horner's algorithm, which requires a total of $\mathcal{O}(K(n + m))$ operations~\cite{TrigFun}.   

The spherefun software allows users to evaluate using either Cartesian or spherical coordinates.  
In the former case,  points that do not satisfy $x^2 + y^2 + z^2 =1$, are projected to the unit sphere in the radial direction. 

\subsubsection*{Computation of Fourier coefficients}\label{sec:FourierCoefficients}
The DFS method and our low rank approximant for $\tilde{f}$ means 
that the FFT is applicable when computing with $\tilde{f}$. Here, we assume that the 
Fourier coefficients for $c_j$ and $r_j$ in~\eqref{eq:FourierExpansions} are unknown. 
In full tensor-product form the bi-periodic BMC-I function can be approximated using a 2D Fourier expansion. That is,
\begin{align}
 \tilde{f}(\lambda,\theta) \approx \sum_{j=-m/2}^{m/2-1}\sum_{k=-n/2}^{n/2-1} X_{jk} e^{ij\theta} e^{ik\lambda}. \label{eq:2DFourierExpansion}
\end{align}
The $m\times n$ matrix $X$ of Fourier coefficients can be directly computed by sampling $\tilde{f}$ on a 2D uniform tensor-product 
grid and using the 2D FFT, costing $\mathcal{O}(mn\log (mn))$ operations.  The low rank structure of 
$\tilde{f}$ allows us to compute a low rank approximation of $X$ in $\mathcal{O}(K(m\log m + n\log n))$ operations from uniform samples of $\tilde{f}$ along 
the adaptively selected skeleton from Section~\ref{sec:lowRankApproximation}.  
%$\tilde{f}$ allows us to compute a low rank approximation of $X$ in $\mathcal{O}(K(m\log m + n\log n))$ operations from uniform samples of $\tilde{f}$ along 
%
%After sampling $\tilde{f}$ along , 
%the coefficients for $c_{j}$ and $r_{j}$ in~\eqref{eq:FourierExpansions}
%are computed in $\mathcal{O}(K^3+K^2(m+n) + K(m\log m + n\log n))$ operations by GE on the 
%skeleton~\cite{townsend2013gaussian} to obtain the values of $c_j$ and $r_j$ at uniform grids and then using the 
%FFT.  The matrix $X$ will be calculated in low rank form 
The matrix $X$ is given in low rank form as $X = A D B^{T}$,
where $A$ is an $m\times K$ matrix and $B$ is an $n\times K$ matrix so that the $j$th column of $A$ and $B$ 
is the vector of Fourier coefficients for $c_j$ and $r_j$, respectively,
and $D$ is a $K\times K$ diagonal matrix containing $d_j$. 
From the low rank format of $X$ one can calculate the entries of $X$ by 
matrix multiplication in $\mathcal{O}(Kmn)$ operations. 

The inverse operation is to sample $\tilde{f}$ on an $m\times n$ uniform grid in $[-\pi,\pi]\times [-\pi,\pi]$ given 
its Fourier coefficient matrix.  If $X$ is given in low rank form, then this can be achieved 
in $\mathcal{O}(K(m\log m + n\log n))$ operations via the inverse FFT. 

These efficient algorithms are regularly employed in spherefun, especially in the Poisson solver (see Section \ref{sec:PoissonSphere}).  The Fourier coefficients of a spherefun object are computed
by the \texttt{coeffs2} command and the values of the function at a uniform $\lambda$-$\theta$ grid are computed by the command \texttt{sample}.

\subsubsection*{Integration}
The definite integral of a function $f(x,y,z)$ over the sphere can be efficiently computed 
in spherefun as follows:
\[
\int_{S}f(x,y,z)dx\,dy\,dz = \int_0^\pi \int_{-\pi}^{\pi} \tilde{f}(\lambda,\theta)\sin\theta \,d\lambda \,d\theta \approx \sum_{j=1}^K d_j \int_0^\pi c_j(\theta)\sin\theta \,d\theta \int_{-\pi}^\pi r_j(\lambda)\,d\lambda
\]
%\[
%\begin{aligned} 
%\int_{S}f(x,y,z)dx\,dy\,dz &= \int_0^\pi \int_{-\pi}^{\pi} \tilde{f}(\lambda,\theta)\sin\theta \,d\lambda \,d\theta \\
%&\approx \sum_{j=1}^K d_j \int_0^\pi c_j(\theta)\sin\theta \,d\theta \int_{-\pi}^\pi r_j(\lambda)\,d\lambda.
%\end{aligned} 
%\]
Hence, the approximation of the integral of $f$ over the sphere reduces 
to $2K$ one-dimensional integrals involving $2\pi$-periodic functions.  

Due to the orthogonality of the Fourier basis, the integrals of $r_j(\lambda)$ are given as
\[
 \int_{-\pi}^\pi r_j(\lambda)\,d\lambda = 2b_0^j, \qquad 1\leq j\leq K,
\]
where $b_0^j$ is the zeroth Fourier coefficient of $r_j$ in~\eqref{eq:FourierExpansions}. 
The integrals of $c_j(\theta)$ are over half the period so the expressions are a bit more complicated.  
Using symmetry and orthogonality, they work out to be 
\begin{equation}
\int_0^\pi c_j(\theta)\sin\theta \, d\theta = \sum_{k=-m/2}^{m/2-1} w_ka_k^j, \qquad 1\leq j\leq K,
\label{eq:integrationCoeffs}
\end{equation} 
where $w_{\pm 1} = 0$ and $w_k = (1+e^{i\pi k})/(1-k^2)$ for $-m/2\leq k\leq m/2-1$  and $k\neq \pm 1$. 
Here, $a_k^j$ are the Fourier coefficients for $c_j$ in~\eqref{eq:FourierExpansions}.

Therefore, we can compute the surface integral of $f(x,y,z)$ over the sphere in 
$\mathcal{O}(Km)$ operations. This algorithm is used in the {\tt sum2} command of 
spherefun.   For example, the function $f(x,y,z) = 1+x+y^2+x^2y+x^4+y^5+(xyz)^2$ has a 
surface integral of $216\pi/35$ and can be calculated in spherefun as follows:
\begin{alltt}
f=spherefun(@(x,y,z) 1+x+y.^2+x.^2.*y+x.^4+y.^5+(x.*y.*z).^2);
sum2(f) 
ans =
    19.388114662154155
\end{alltt} 
The error is computed as \texttt{abs(sum2(f)-216*pi/35)} and is given by $3.553\times10^{-15}$.

%This surprising efficiency extends to a broad 
%range of tasks, which includes all the basic operations. 

% \begin{definition}
%  A tensor product operator $\mathcal{L}$ is a linear operator on functions $\lambda$ and 
%  $\theta$ with the property that if $f(\lambda,\theta) = d_1c_1(\theta)r_1(\lambda)$, then
%  $\mathcal{L}(f) = d_1\mathcal{L}_\theta(c_1)\mathcal{L}_\lambda(r_1)$, for some operators 
%  $\mathcal{L}_\theta$ and $\mathcal{L}_\lambda$. Moreover, if $f$ has rank $k$, then 
%  \[
%   \mathcal{L}\left(\sum_{j=1}^k d_j c_j(\theta)r_j(\lambda)\right) = \sum_{j=1}^k d_j \left(\mathcal{L}_\theta c_j(\theta)\right)\left(\mathcal{L}_\lambda r_j(\lambda)\right).
%  \]
% \end{definition}
% 
% Partial differentiation, $\partial f/\partial x$, is not strictly a tensor product operation 
% because $\partial f/\partial x$ can have a higher rank than $f$. However, it is the same 
% of two tensor product operations. 

\subsubsection*{Differentiation} \label{sec:diff}
Differentiation of a function on the sphere with respect to spherical coordinates $(\lambda,\theta)$ can lead to 
singularities at the poles, even for smooth functions~\cite{Swarztrauber81}.  For example, consider the 
simple function $f(\lambda,\theta) = \cos\theta$.  The $\theta$-derivative of this function is $\sin\theta$, which is continuous 
on the sphere but not smooth at the poles.  Fortunately, one is typically interested in the derivatives that arise 
in applications such as in vector calculus operations involving the gradient, divergence, curl, or Laplacian.  
All of these operators can be expressed in terms of the components of the surface gradient with 
respect to the Cartesian coordinate system~\cite{FlyerWright}.   

Let $\mathbf{e}^{x}$, $\mathbf{e}^{y}$, and $\mathbf{e}^{z}$, denote the unit vectors in the 
$x$, $y$, and $z$ directions, respectively, and $\grads$ denote the surface gradient on the sphere in Cartesian coordinates.
From the chain rule, we can derive the Cartesian components of  $\grads$  as
\begin{align}
\mathbf{e}^{x}\cdot\grads := \frac{\tpartial }{\partial x} &= -\frac{\sin \lambda}{\sin \theta}\frac{\partial}{\partial \lambda} + \cos\lambda\cos\theta\frac{\partial}{\partial \theta}, \label{eq:diff_x} \\
\mathbf{e}^{y}\cdot\grads := \frac{\tpartial }{\partial y} &= \phantom{-}\frac{\cos \lambda}{\sin \theta}\frac{\partial}{\partial \lambda} + \sin\lambda\cos\theta\frac{\partial}{\partial \theta}, \label{eq:diff_y} \\
\mathbf{e}^{z}\cdot\grads := \frac{\tpartial }{\partial z} &= \phantom{-\frac{\sin \lambda}{\sin \theta}\frac{\partial}{\partial \lambda} + \cos\lambda}\sin\theta\frac{\partial}{\partial \theta}. \label{eq:diff_z}
\end{align}
Here, the superscript `${\rm t}$' indicates that these operators are tangential gradient operations. 
The result of applying any of these operators to a smooth function on the sphere is a smooth function on 
the sphere~\cite{Swarztrauber81}.  For example, applying $\partial^t / \partial x$ to $\cos\theta$ gives  
$-\cos\lambda\,\sin\theta\,\cos\theta$, which in Cartesian coordinates is $-xz$ restricted to the sphere.  

%There is an abuse of notation here as we are referring to these operators as partial derivatives with respect to $x$, $y$, and $z$.  Strictly speaking, they 
%represent the components of the tangential gradient of a function on the surface of the sphere in the $x$, $y$, and $z$ 
%directions, respectively.  
%We prefer the notation $\partial / \partial x$, $\partial / \partial y$, and $\partial / \partial z$ because the standard 
%vector calculus operations can be conveniently expressed in terms of these definitions (see Section~\ref{sec:vectorCalculus}).

As with integration, our low rank approximation for $\tilde{f}$ can be exploited to compute 
\eqref{eq:diff_x}--\eqref{eq:diff_z} efficiently. For example, using~\eqref{eq:FourierExpansions} we have
\begin{equation}
\begin{aligned}
 \frac{\tpartial \tilde{f}}{\partial x} &= -\frac{\sin\lambda}{\sin\theta}\frac{\partial \tilde{f}}{\partial \lambda} + \cos\lambda\cos\theta\frac{\partial \tilde{f}}{\partial \theta}\\
 &\approx -\sum_{j=1}^K \left(\frac{c_j(\theta)}{\sin\theta}\right)\left(\sin\lambda \, \frac{\partial r_j(\lambda)}{\partial \lambda}\right) + \sum_{j=1}^K \left(\cos\theta \, \frac{\partial c_j(\theta)}{\partial \theta}\right)\left(\cos\lambda\, r_j(\lambda)\right).\\
\end{aligned}
\label{eq:dfdx}
\end{equation}
It follows that $\tpartial \tilde{f}/\partial x$ can be calculated by essentially 1D 
algorithms involving differentiating Fourier expansions as well as 
multiplying and dividing them by cosine and sine.  In the above expression, for example, we have 
\[
 (\sin\lambda)\frac{\partial r_j(\lambda)}{\partial \lambda} = \!\!\!\sum_{k=-n/2-1}^{n/2}\!\!\!\!\! \tfrac{-(k+1)b_{k+1}^j+(k-1)b_{k-1}^j}{2}  e^{i k \lambda}, \quad  (\cos\lambda) r_j(\lambda) = \!\!\!\sum_{k=-n/2-1}^{n/2}\!\!\!\!\! \tfrac{b_{k+1}^j+b_{k-1}^j}{2} e^{i k \lambda},
\]
%\[
% (\sin\lambda)\frac{\partial r_j(\lambda)}{\partial \lambda} = \sum_{k=-n/2-1}^{n/2} \frac{-(k+1)b_{k+1}^j+(k-1)b_{k-1}^j}{2}  e^{i k \lambda}
%\]
%and 
%\[
% (\cos\lambda) r_j(\lambda) = \sum_{k=-n/2-1}^{n/2} \frac{b_{k+1}^j+b_{k-1}^j}{2} e^{i k \lambda},
%\]
where $b_{-n/2-2}^{j} = b_{-n/2-1}^{j} = 0$ and $b_{n/2}^{j} = b_{n/2+1}^{j}= 0$.  Note that the number of coefficients in the Fourier representations of these derivatives has 
increased by two modes 
to account for multiplication by $\sin\lambda$ and $\cos\lambda$.  Similarly, we also have
\[
 (\cos\theta)\frac{\partial c_j(\theta)}{\partial \theta} = \sum_{k=-m/2-1}^{m/2+1} \frac{(k+1)ia_{k+1}^j+(k-1)ia_{k-1}^j}{2}  e^{i k \theta},
\]
where $a_{-m/2-2}^{j} = a_{-m/2-1}^{j} = 0$ and $a_{m/2}^{j} = a_{m/2+1}^{j} = 0$.  Lastly, 
for~\eqref{eq:dfdx} we must compute $c_j(\theta)/\sin\theta$. This can be done as follows:
\begin{equation}
  \frac{c_j(\theta)}{\sin\theta} = \sum_{k=-m/2}^{m/2-1} (M_{\sin}^{-1}\underline{a}^j)_k  e^{i k \lambda}, \qquad  
  M_{\sin} = \frac{i}{2}\begin{bmatrix}0 & 1\cr -1 & 0 & \ddots \cr & \ddots & \ddots & 1 \cr && -1 & 0\end{bmatrix},
\label{eq:Msin}
\end{equation}
where $\underline{a}^{j} = (a_{-m/2}^{j},\ldots,a_{m/2-1}^{j})^T$. Here, $M_{\sin}^{-1}$ exists because $m$ is 
an even integer.\footnote{The eigenvalues of $M_{\sin}$ are $\cos(\pi \ell/(m+1))$, $1\leq \ell\leq m$, which are all nonzero when $m$ is even.}

Therefore, though it appears that~\eqref{eq:dfdx} is introducing 
an artificial pole singularity by the division of $\sin\theta$, this 
is not case.  Our treatment of the artificial pole singularity by operating on the 
coefficients directly appears to be new.  The standard technique when 
using spherical coordinates on a latitude-longitude grid is to shift 
the grid in the latitude direction so that the poles are not 
sampled~\cite{Fornberg_95_01,Cheong2000261,Yee1981}. In~\eqref{eq:Msin} 
there is no need to explicitly avoid the pole, it is easy to implement, 
and is possibly more accurate numerically than shifting the grid. This 
algorithm costs $\mathcal{O}(K(m+n))$ operations.
  
We use similar ideas to compute \eqref{eq:diff_y} and \eqref{eq:diff_z}, which require
a similar number of operations. They are implemented in the \texttt{diff} command of spherefun. 

\subsection{Vector-valued functions on the sphere and vector calculus}\label{sec:vectorCalculus}
Expressing vector-valued functions that are tangent to the sphere in spherical coordinates is very inconvenient
since the unit vectors in this coordinate system are singular at the poles~\cite{Swarztrauber81}.   
It is therefore common practice to express vector-valued functions in Cartesian coordinates, not latitude--longitude coordinates.  
In Cartesian coordinates the components of the vector field are smooth and can be approximated using the 
low rank techniques developed in Section~\ref{sec:lowRankApproximation}.  

All the standard vector-calculus operations can be expressed in terms of the tangential derivative operators 
in~\eqref{eq:diff_x}--\eqref{eq:diff_z}.   For example the surface gradient, $\grads$, of a scalar-valued
function $f$ on the sphere is given by the vector
\begin{align*}
\grads f = \begin{bmatrix} \diffx{f}, & \diffy{f}, & \diffz{f} \end{bmatrix}^{T},
\end{align*}
where the partial derivatives are defined in \eqref{eq:diff_x}--\eqref{eq:diff_z}.  The surface divergence and curl of a vector 
field $\mathbf{f} = \begin{bmatrix} f_1, & f_2, & f_3 \end{bmatrix}^T$ that is tangent to the sphere 
can also be computed using \eqref{eq:diff_x}--\eqref{eq:diff_z} as
\[
\grads \cdot \mathbf{f} = \diffx{f_1} + \diffy{f_2} + \diffz{f_3},\;  \grads \times \mathbf{f} = 
\begin{bmatrix} 
\diffy{f_3} - \diffz{f_2}, & 
\diffz{f_1} - \diffx{f_3}, & 
\diffx{f_2} - \diffy{f_1}
\end{bmatrix}^{T}.
\]
The result of the surface curl $\grads \times \mathbf{f}$ is a vector that is tangent to the sphere.

In 2D one can define the ``curl of a scalar-valued function'' as the cross product of the unit normal vector to the surface 
and the gradient of the function. For a scalar-valued function on the sphere, the curl in Cartesian coordinates is given by
\begin{align}
\mathbf{n}\times\grads f = 
\begin{bmatrix}
y \diffz{f} - z \diffy{f}, & 
z \diffx{f} - x \diffz{f}, &
x \diffy{f} - y \diffx{f} 
\end{bmatrix}^{T},
\label{eq:CurlScalarF}
\end{align}
where $x$, $y$, and $z$ are points 
on the unit sphere given by~\eqref{eq:sphericalCoordinates}.  This follows from the fact that  
the unit normal vector at $(x,y,z)$ on the unit sphere is just $\mathbf{n} = (x,y,z)^T$.

The final vector calculus operation we consider is the vorticity of a vector field, which for a two-dimensional surface is a 
scalar-valued function defined as $\zeta = ( \grads \times \mathbf{f})\cdot\mathbf{n}$, and can 
be computed based on the operators described above.
%the dot product of the unit normal vector to the surface, $\mathbf{n}$, with the surface curl of the vector field.  
%This can be computed with respect to Cartesian coordinates using the descriptions 
%for the other operators written above.

Vector-valued functions are represented by \texttt{spherefunv} objects, which contain three 
spherefun objects, one for each component of the vector-valued function. Low rank techniques described in Section~\ref{sec:lowRankApproximation} 
are employed on each component separately. The operations listed above can be computed using the \texttt{grad},  
\texttt{div}, \texttt{curl}, and \texttt{vort} commands; see Figure~\ref{fig:VectorCalculus} for an 
example.

\begin{figure} 
\begin{center}
\includegraphics[width=0.33\textwidth]{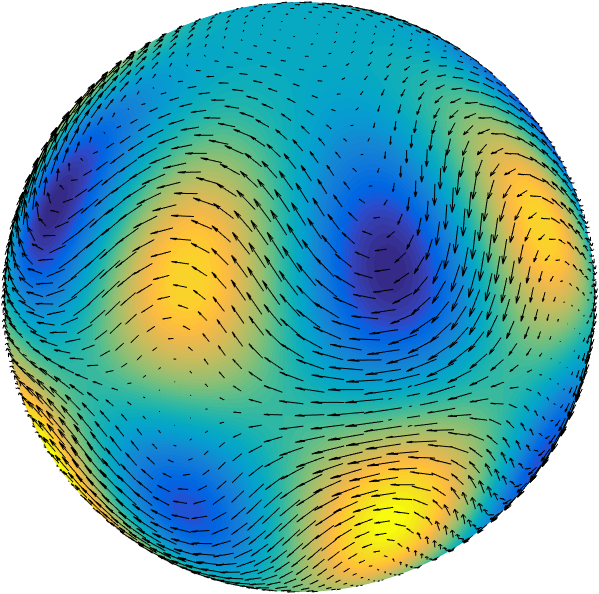}
\end{center}
 \caption{Arrows indicate the tangent vector field generated from $\mathbf{u} = \grads \times \psi$, where 
 $\psi(\lambda,\theta)=\cos\theta + (\sin\theta)^4\cos\theta\cos(4\lambda)$, which is the stream function 
 for the Rossby--Haurwitz benchmark problem for the shallow water wave equations~\cite{williamson1992}.  
 After constructing $\psi$ in spherefun, the tangent vector field was computed using \texttt{u = curl(psi)}, and plotted using 
 \texttt{quiver(u)}.  The superimposed false color plot represents the vorticity of $\mathbf{u}$ and is computed using \texttt{vort(u)}.}
 \label{fig:VectorCalculus} 
\end{figure}

\subsection*{Miscellaneous operations}
The spherefun class is written as part of Chebfun, which means that spherefun objects have immediate access to all the operations 
available in Chebfun.  For operations that do not require a strict adherence to the symmetry of the sphere, we 
can use Chebfun2 with spherical coordinates~\cite{Chebfun2}.  Important examples include two-dimensional optimization such 
as {\tt min2}, {\tt max2}, and {\tt roots} as well as continuous linear algebra operators such as {\tt lu} and {\tt flipud}.  
The operations that use the Chebfun2 technology are performed seamlessly without user intervention. 

\section{A fast and optimal spectral method for Poisson's equation}\label{sec:PoissonSphere}
The DFS method leads to an efficient spectral method for solving 
Poisson's equation on the sphere. 
The Poisson solver that we describe is simple, based on the Fourier spectral method, 
and has optimal complexity. Related approaches can be found in~\cite{Yee1981,shen1999efficient,Cheong2000327}.

Given a function $f$ on the sphere satisfying $\int_0^\pi\int_{-\pi}^\pi  f(\lambda,\theta) \sin\theta d\lambda d\theta = 0$, 
Poisson's equation in spherical coordinates is given by
\begin{equation}
 (\sin\theta)^2 \frac{\partial^2 u}{\partial \theta^2}+ \cos\theta\sin\theta \frac{\partial u}{\partial \theta} + \frac{\partial^2 u}{\partial \lambda^2} = (\sin\theta)^2f,\qquad (\lambda,\theta)\in[-\pi,\pi]\times[0,\pi].
\label{eq:SphericalLaplacian}
\end{equation}
Due to the integral condition on $f$,~\eqref{eq:SphericalLaplacian} has infinitely 
many solutions, all differing by a constant. 
To fix this constant it is standard to solve~\eqref{eq:SphericalLaplacian} together with the 
constraint  $\int_0^\pi\int_{-\pi}^\pi  u(\lambda,\theta) \sin\theta d\lambda d\theta = 0.$ 
Here, we assume that this constraint is imposed.
 
One can solve~\eqref{eq:SphericalLaplacian} directly on the domain $[-\pi,\pi]\times [0,\pi]$,
but then the solution $u$ is not $2\pi$-periodic in the $\theta$-variable due to the coordinate transform. 
To recover the periodicity in $\theta$, we use the DFS method (see Section~\ref{sec:doubleFouriersphere}) 
and seek an approximation for the solution, denoted by $\tilde{u}$, to the ``doubled-up'' version of~\eqref{eq:SphericalLaplacian} 
given by
\begin{equation} 
  (\sin\theta)^2 \tilde{u}_{\theta\theta} + \cos\theta\sin\theta \tilde{u}_\theta + \tilde{u}_{\lambda\lambda} = (\sin\theta)^2\tilde{f},\qquad (\lambda,\theta)\in[-\pi,\pi]^2,
\label{eq:BMCPoisson}
\end{equation} 
where $\tilde{f}$ is a BMC-I function that is bi-periodic, see~\eqref{eq:BMCsphere}.
One can verify that the solution $\tilde{u}$ to~\eqref{eq:BMCPoisson} must 
also be a BMC-I function, i.e., a continuous function on the sphere. 
On the domain $[-\pi,\pi]\times[0,\pi]$ the solution 
$\tilde{u}$ in~\eqref{eq:BMCPoisson} must 
coincide with the solution $u$ to~\eqref{eq:SphericalLaplacian}. Thus, 
we impose the same integral constraint on $\tilde{u}$: 
\begin{equation}
 \int_0^\pi\int_{-\pi}^\pi \tilde{u}(\lambda,\theta) \sin\theta d\lambda d\theta = 0. 
\label{eq:integralConstraint2}
\end{equation} 

Since all the functions in~\eqref{eq:BMCPoisson} are bi-periodic,  
we discretize the equation by the Fourier spectral method~\cite{boyd2001chebyshev}, 
and $\tilde{u}$ is represented by a 2D Fourier expansion, i.e., 
\begin{equation}
\tilde{u}(\lambda,\theta) \approx \sum_{j=-m/2}^{m/2-1}\sum_{k=-n/2}^{n/2-1} X_{jk} e^{i j \theta}e^{i k \lambda}, \qquad (\lambda,\theta)\in[-\pi,\pi]^2,
\label{eq:tildeuFourierExpansion}
\end{equation}
where $m$ and $n$ are even integers, and seek to compute the coefficient matrix $X\in\mathbb{C}^{m\times n}$.  
Continuous operators, such as differentiation and multiplication, are 
now discretized to matrices by carefully inspecting how each operation 
modifies the coefficient matrix $X$ in~\eqref{eq:tildeuFourierExpansion} and representing 
the action by a matrix. For example,
\[
 \frac{\partial \tilde{u}}{\partial \theta} = \!\!\!\sum_{j=-m/2}^{m/2-1}\sum_{k=-n/2}^{n/2-1}\!\!\!\! jiX_{jk} e^{i j \theta}e^{i k \lambda}, \quad (\cos\theta)\tilde{u} = \!\!\!\sum_{j=-m/2}^{m/2-1}\sum_{k=-n/2}^{n/2-1} \!\!\!\!\frac{X_{j+1,k}+X_{j-1,k}}{2} e^{i j \theta}e^{i k \lambda},
\]
where $X_{m/2+1,k} = 0$ and $X_{-m/2,k} = 0$ for $-n/2-1\leq k\leq n/2$. Thus, 
we can represent $\partial/\partial\theta$ and multiplication 
by $\cos\theta$ by $D_mX$ and $M_{\cos}X$, respectively, where
%\[
% D_m = \begin{pmatrix}-\tfrac{mi}{2} \\[-2pt] & \ddots \\[-2pt] &&-i \\[-3pt]&&& 0\\[-2pt]&&&& i \\[-2pt]&&&&& \ddots \\[-2pt]&&&&&& \frac{(m-2)i}{2}\end{pmatrix}, \quad M_{\cos} = \begin{pmatrix}0 & \tfrac{1}{2}\cr \tfrac{1}{2} & 0 & \tfrac{1}{2} \cr & \tfrac{1}{2} & \ddots & \ddots & \cr & & \ddots & \ddots & \tfrac{1}{2}\cr  &&& \tfrac{1}{2} & 0 & \tfrac{1}{2} \cr &&&& \tfrac{1}{2} & 0\end{pmatrix}.
%\]
\[
 D_m = \text{diag}\left(\left[-\tfrac{mi}{2}, \cdots,\;  -i,\;  0,\;  i, \cdots,\; \tfrac{(m-2)i}{2}\right]\right), \quad 
 M_{\cos} = \frac{1}{2}\begin{bmatrix}0 & 1\cr 1 & 0 & \ddots \cr & \ddots & \ddots & 1 \cr && 1 & 0\end{bmatrix}.
\]
Similar reasoning shows that $\partial/\partial\lambda$ and multiplication 
by $\sin\theta$ can be discretized as $XD_n$ and $M_{\sin}X$, where $M_{\sin}$ 
is given in~\eqref{eq:Msin}.
Therefore, we can discretize~\eqref{eq:BMCPoisson} by the following 
Sylvester matrix equation:  
\begin{equation} 
 \left(M_{\sin}^2D_m^2 + M_{\cos}M_{\sin}D_m\right)X + XD_n^2 = F,
 \label{eq:SylvesterSphericalPoisson} 
\end{equation} 
where $F\in\mathbb{C}^{m\times n}$ is the matrix of 2D Fourier coefficients 
for $(\sin\theta)^2\tilde{f}$ in an expansion like~\eqref{eq:tildeuFourierExpansion}.

We note that~\eqref{eq:SylvesterSphericalPoisson} can be solved very 
fast because $D_n$ is a diagonal matrix and hence, each column of $X$ can 
be found independently of the others. Writing 
$X = \left[X_{-n/2}\,|\,\cdots\,|\,X_{n/2-1}\right]$ and 
$F = \left[F_{-n/2}\,|\,\cdots\,|\,F_{n/2-1}\right]$, we can equivalently 
write~\eqref{eq:SylvesterSphericalPoisson} as $n$ decoupled linear systems,
\begin{equation} 
  \left(M_{\sin}^2D_m^2 + M_{\cos}M_{\sin}D_m - (D_n^2)_{kk}I_m\right)X_k = F_k, \qquad -n/2\leq k\leq n/2-1,
\label{eq:linearSystems}
\end{equation} 
where $I_m$ denotes the $m\times m$ identity matrix. 

For $k\neq 0$, the linear systems in~\eqref{eq:linearSystems} have a 
pentadiagonal structure and are invertible. They can be solved 
by backslash, i.e., `\textbackslash', in MATLAB 
that employs a sparse LU solver. For each $k\neq0$ this requires 
just $\mathcal{O}(m)$ operations, for a total of $\mathcal{O}(mn)$ 
operations for the linear systems in~\eqref{eq:linearSystems} with 
$-n/2\leq k\leq n/2-1$ and $k\neq 0$. 

%When $k=0$ the linear system in~\eqref{eq:linearSystems} is not invertible. This makes sense because 
%without imposing the integral constraints in~\eqref{eq:integralConstraint2}, the equation in~\eqref{eq:BMCPoisson} 
%has infinitely many solutions that differ by constants. We must use 
%the integral constraint in~\eqref{eq:integralConstraint2} when $k=0$. 
%To discretize the constraint we note that 
When $k=0$ the linear system in~\eqref{eq:linearSystems} is not invertible because we have not accounted for the
integral constraint in~\eqref{eq:integralConstraint2}, which fixes the free constant the solutions can differ by.  We account for
this constraint on the $k=0$ mode by noting that
\[
 \int_0^\pi\int_{-\pi}^\pi \tilde{u}(\lambda,\theta)\sin\theta d\lambda d\theta \approx  2\pi\sum_{j=-m/2}^{m/2-1} X_{j0} \frac{1+e^{i\pi j}}{1-j^2},
\]
%\[
%\begin{aligned} 
% \int_0^\pi\int_{-\pi}^\pi \tilde{u}(\lambda,\theta)\sin\theta d\lambda d\theta &\approx \sum_{j=-n/2}^{\tfrac{m}{2}-1}\sum_{k=-n/2}^{n/2-1} X_{jk} \int_0^\pi \sin\theta e^{i j \theta}d\theta \int_{-\pi}^\pi e^{i k \lambda}d\lambda= 2\pi\sum_{j=-\tfrac{m}{2}}^{\tfrac{m}{2}-1} X_{j0} \frac{1+e^{i\pi j}}{1-j^2}\\
% & = 2\pi\sum_{j=-\tfrac{m}{2}}^{\tfrac{m}{2}-1} X_{j0} \int_0^\pi \sin\theta e^{i j \theta}d\theta = 2\pi\sum_{j=-\tfrac{m}{2}}^{\tfrac{m}{2}-1} X_{j0} \frac{1+e^{i\pi j}}{1-j^2}.\\
%\end{aligned}
%\]
which can be written as 
$2\pi w^TX_0 = 0$, where the vector $w$ is given in~\eqref{eq:integrationCoeffs}. 
We impose $2\pi w^TX_0 = 0$ on $X_0$ by replacing 
the zeroth row of the linear system $(M_{\sin}^2D_m^2 + M_{\cos}M_{\sin}D_m)X_0 = F_0$ 
with $2\pi w^TX_0 = 0$. We have selected the zeroth row 
because it is zero in the linear system.  
Thus, we solve the following linear system:
\begin{equation} 
 \begin{bmatrix} 
  \quad   w^T   \quad    \cr
   P\left(M_{\sin}^2D_m^2 + M_{\cos}M_{\sin}D_m\right)
 \end{bmatrix}
 X_0
 = 
  \begin{bmatrix} 
  0 \\[3pt]
  PF_0
  \end{bmatrix},
\label{eq:pentdiagonalPlusRankOne}
\end{equation} 
where $P\in\mathbb{R}^{(m-1)\times m}$ is a projection matrix that removes the zeroth row, i.e., 
\[
 P\left(v_{-m/2},\ldots, v_{-1}, v_0, v_{1},\ldots, v_{m/2-1}\right)^T = \left(v_{-m/2},\ldots, v_{-1}, v_{1},\ldots, v_{m/2-1}\right)^T.
\]

The linear system in~\eqref{eq:pentdiagonalPlusRankOne} is banded with one dense 
row, which can be solved in $\mathcal{O}(m)$ operations using the adaptive QR 
algorithm~\cite{Olver_13_01}. For simplicity, since solving~\eqref{eq:pentdiagonalPlusRankOne} is not 
the dominating computational cost we use the backslash command  
in MATLAB on sparse matrices, which requires $\mathcal{O}(m)$ 
operations. 

The resulting Poisson solver may be regarded as having an optimal complexity of
$\mathcal{O}(mn)$ because we solve for $mn$ Fourier coefficients in~\eqref{eq:tildeuFourierExpansion}. 
In practice, one may need to calculate the matrix of 2D Fourier 
coefficients for $\tilde{f}$ that costs $\mathcal{O}(mn\log (mn))$ operations 
if the low rank approximation of $\tilde{f}$ is not exploited. If the low rank 
structure of $\tilde{f}$ is exploited, then since the whole $m\times n$ matrix 
coefficients $F$ is required in the Poisson solver the cost is $\mathcal{O}(mn)$ operations (see Section~\ref{sec:FourierCoefficients}).

In Figure~\ref{fig:FastSphericalPoissonSolver} (left) the solution to $\nabla^2u = \sin(50xyz)$ 
on the sphere is shown. Here, we used our algorithm with $m=n=150$. 
Before we can apply the algorithm, the matrix of 2D Fourier coefficients for $\sin(50xyz)$ is computed. Since 
the BMC-I function associated with $\sin(50xyz)$ has a
numerical rank of $12$ this costs $\mathcal{O}(mn)$ operations. In Figure~\ref{fig:FastSphericalPoissonSolver} (right)
we verify the complexity of our Poisson solver by showing timings for $m=n$. We have denoted the number of 
degrees of freedom of the final solution as $mn/2$ since this is the number that is employed on the solution 
$u$. Without explicit parallelization, even though the solver is embarrassingly parallel, we can solve for 
$100$ million degrees of freedom in the solution in one minute on a standard laptop.\footnote{Timings were done on a MacBook Pro using MATLAB 2015b without explicit parallelization.}

\begin{figure} 
 \begin{minipage}{.49\textwidth} 
 \begin{overpic}[width=.9\textwidth,trim=120 95 120 80,clip]{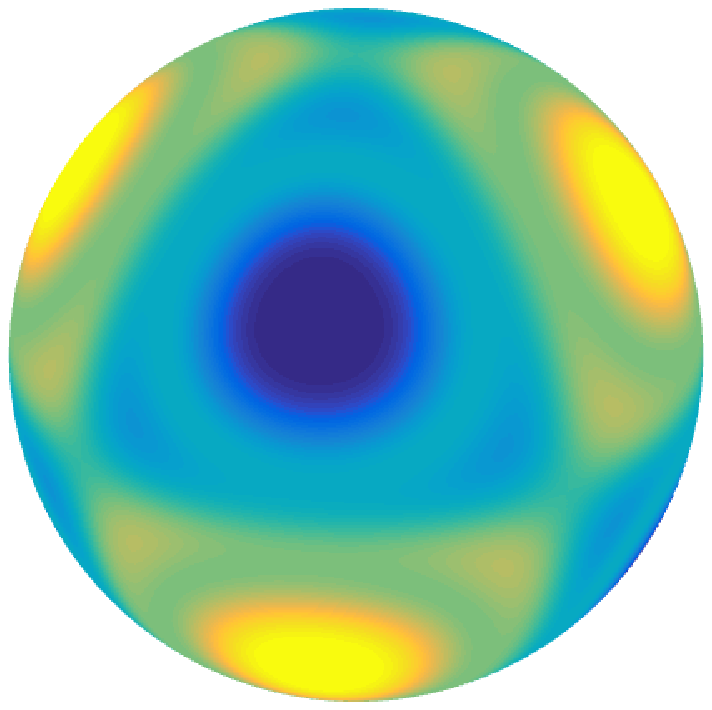}
  
 \end{overpic}
 \end{minipage}
 \begin{minipage}{.49\textwidth} 
 \begin{overpic}[width=.9\textwidth]{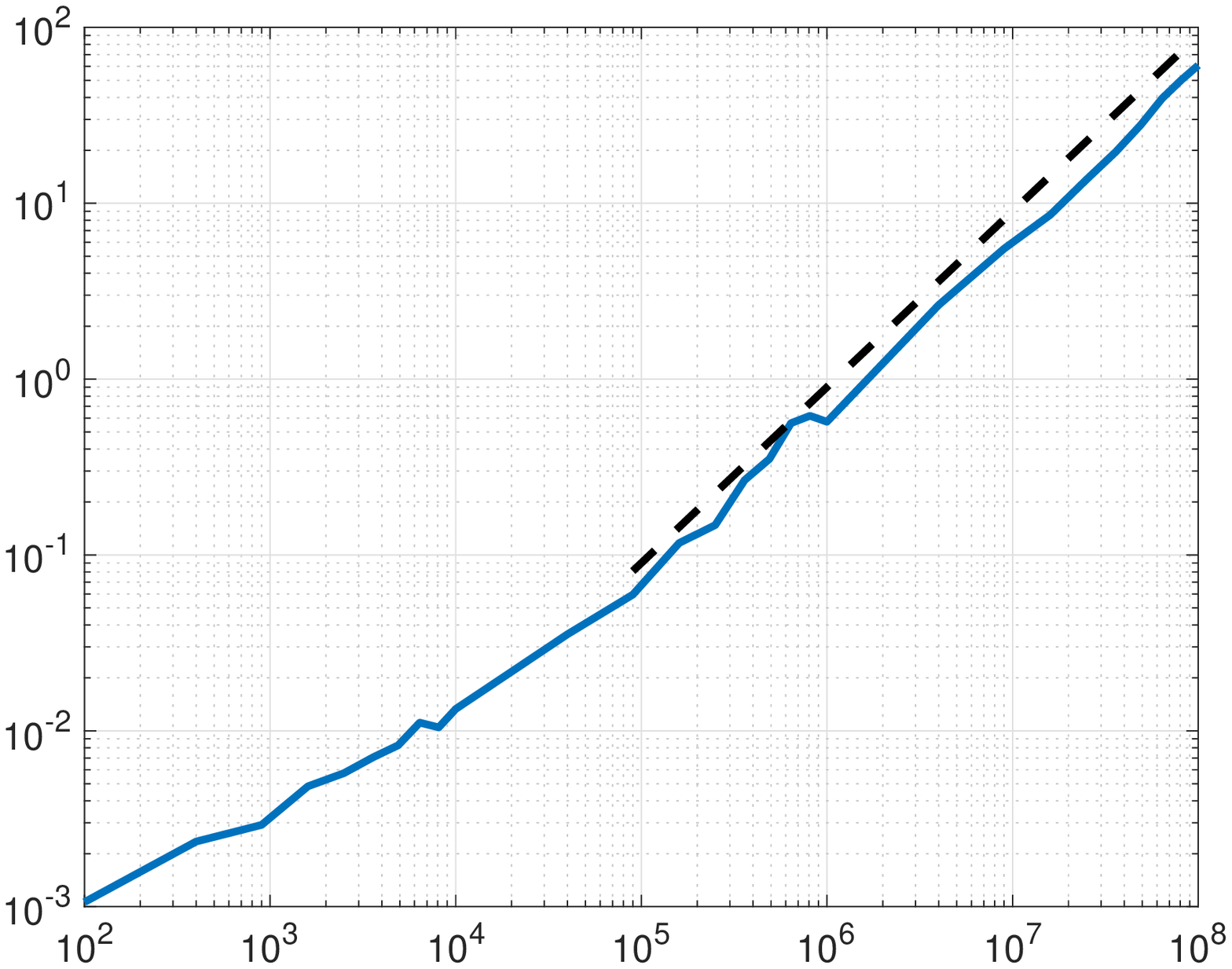}
  \put(-5,15) {\rotatebox{90}{{\small Computational time (sec)}}}
  \put(45,-3) {{\small $m/2 \times n$}}
  \put(60,45) {\rotatebox{48}{$\mathcal{O}(mn)$}}
 \end{overpic}
 \end{minipage}
 \caption{Left: Solution to $\nabla^2u=\sin(50xyz)$ with a zero 
 integral constraint computed by {\tt f = spherefun(@(x,y,z) sin(50*x.*y.*z)); u = Poisson(f,0,150,150);}, 
 which employs the algorithm above with $m=n=150$.
 Right: Execution time of the Poisson solver as a function of the number of unknowns, $nm/2$, when $m = n$.}
 \label{fig:FastSphericalPoissonSolver}
\end{figure}

\section*{Conclusions} 
The double sphere method is synthesized with low 
rank approximation techniques to develop a software system for 
computing with functions on the sphere to essentially 
machine precision. We show how 
symmetries in the resulting functions can be preserved by an iterative variant 
of Gaussian elimination to efficiently construct low rank approximants. A collection of fast algorithms
are developed for differentiation, integration, 
vector calculus, and solving Poisson's equation. Now an
investigator can compute with functions on the sphere without worrying about the underlying discretizations. 
The code is publicly available as part of Chebfun~\cite{Chebfun}.

\section*{Acknowledgments}
We would like to congratulate Nick Trefethen on his 60th birthday and celebrate his 
impressive contribution to the field of scientific computing.  We dedicate this paper to him. 
We thank Behnam Hashemi and Hadrien Montanelli from the University of Oxford 
for reviewing the spherefun code and Zack Lindbloom-Brown and Jeremy Upsal from the University of Washington for 
their explorations with spherefun. We also thank the editor and referees for their valuable comments.

\bibliographystyle{siam}
\bibliography{spherefun}

\end{document}